%% file: ctables.tex
\documentclass{imsart}
\RequirePackage[OT1]{fontenc}
\RequirePackage{amsthm,amsmath}
\RequirePackage[numbers]{natbib}

\usepackage{amssymb, bbm, amscd}

\usepackage{float}
\usepackage{graphicx}
\usepackage{algorithm,algorithmic}
\usepackage{parskip}
\usepackage[euler]{textgreek}
\usepackage[T1]{fontenc}
\usepackage[utf8]{inputenc}

\usepackage{lipsum} % for filler text

\startlocaldefs

\newcommand\E{\mathbb E}
\newcommand\I{\mathbbm1}
\renewcommand\P{\mathbb P}
\renewcommand\Pr{\mathbb P}
\newcommand\R{\mathbb R}
\newcommand\N{\mathbb N}

\newcommand\cF{\mathcal F}
\newcommand\cL{\mathcal L}

\newcommand\cX{\mathcal A}
\newcommand\cY{\mathcal B}
\newcommand\cZ{ E}

\def\e{\mathbb{E}\, }
\def\L{\mathcal{L}}
\def\R{\mathbb{R}}
\def\I{\mathcal{I}}
\def\J{\mathcal{J}}

\def\X{ \textbf{X} }

\def\Y{ \textbf{Y} }

\def\XI{\textbf{$X^{(I)}$}}
\def\xI{\textbf{$x^{(I)}$}}
\def\EI{E^{(I)}}
\def\yI{y_{I}}

\def\return{\textbf{return} }
\def\restart{\textbf{restart}}
\def\assume{\textbf{Assumption:} }
\def\range{ \text{range}} 

\newcommand\ignore[1]{ }

\newcommand\RESTART{\STATE \textbf{restart} }

\DeclareMathOperator{\Var}{Var}

\DeclareMathOperator{\Geo}{Geo}
\DeclareMathOperator{\NB}{NB}

\DeclareMathOperator{\Bern}{Bern}

\newtheorem{theorem}{Theorem}[section]
\newtheorem{lemma}[theorem]{Lemma}
\newtheorem{conjecture}[theorem]{Conjecture}
\newtheorem{corollary}[theorem]{Corollary}
\newtheorem{proposition}[theorem]{Proposition}
\theoremstyle{definition}
\newtheorem{definition}{Definition}[section]
\newtheorem{remark}[theorem]{Remark}
\newtheorem{assumption}[theorem]{Assumption}

\begingroup
    \makeatletter
    \@for\theoremstyle:=definition,remark,plain\do{%
        \expandafter\g@addto@macro\csname th@\theoremstyle\endcsname{%
            \addtolength\thm@preskip\parskip
            }%
        }
\endgroup

\endlocaldefs

\begin{document}

\begin{frontmatter}

\title{Random Sampling of Contingency Tables via Probabilistic Divide-and-Conquer}
\runtitle{Contingency Tables via PDC}

\begin{aug}
  \author{\fnms{Stephen}  \snm{DeSalvo}\corref{}\thanksref{t1}\ead[label=e1]{stephendesalvo@math.ucla.edu}}
  \and
  \author{\fnms{James Y.} \snm{Zhao}\thanksref{t2}\ead[label=e2]{james.zhao@usc.edu}}

\runauthor{S.~DeSalvo and J.~Y.~Zhao}

  \thankstext{t1}{UCLA Department of Mathematics, stephendesalvo@math.ucla.edu}
  \thankstext{t2}{USC Department of Mathematics, james.zhao@usc.edu}

%  \runauthor{Stephen DeSalvo}

%  \affiliation{University of California, Los Angeles and University of Southern California}
\ignore{
  \address{UCLA Department of Mathematics \\
  		520 Portola Plaza \\
  		Los Angeles, CA 90095 \\
          \printead{e1}} 
  \address{USC Department of Mathematics \\
  		3620 S. Vermont Ave., KAP 104 \\
		Los Angeles, CA 90089-2532,\\
          \printead{e2}}
}
\end{aug}
%\author[1]{Stephen DeSalvo\thanks{stephendesalvo@math.ucla.edu}}
%\author[2]{James Y. Zhao\thanks{james.zhao@usc.edu}}
%\affil[1]{Department of Mathematics, University of California, Los Angeles}
%\affil[2]{Department of Mathematics, University of Southern California}

\date{January 7, 2016}

\begin{abstract}
We present a new approach for random sampling of contingency tables of any size and constraints based on a recently introduced \emph{probabilistic divide-and-conquer} technique. 
%For an $m \times n$ table, with largest row sum or column sum denoted by $M$, %and the sum of all entries in the table denoted by $N$, 
%the total expected runtime cost to sample a nonnegative integer-valued table uniformly from the set of contingency tables is at most  
%\[O\left(\log(m\, n\, M)\,m^2\,n\right) \times O\left(m\, n\, M\, \log^2(M)\right). \] 
%where $M$ is the largest row sum or column sum, $\Delta>1$, and $s_0$ is the cost to compute certain rejection probabilities. 
%The same algorithm applies, with one extra step, for contingency tables with real--valued entries. 
A simple exact sampling algorithm is presented for %$\{0,1\}$-valued tables, and several alternative algorithms are presented for 
$2\times n$ tables, as well as a generalization % and for the general case 
where each entry of the table has a specified marginal distribution.  %, and also for $\{0,1\}$-
%\smallskip
%\noindent \textbf{Keywords.} Contingency tables, exact sampling, probabilistic divide-and-conquer, transportation polytope

%\smallskip
%\noindent \textbf{MSC classes:} 62H17, 60C05, 52B99
\end{abstract}



\end{frontmatter}

\maketitle

\ignore{
\def\spacingset#1{\renewcommand{\baselinestretch}%
{#1}\small\normalsize} \spacingset{1}

\parskip 10pt

\begin{center}
\LARGE Exact, Uniform Sampling of Contingency Tables\\[.3cm]
\normalsize Stephen DeSalvo and James Y. Zhao\\[.1cm]
November 10, 2015\\[.5cm]
\parbox{.8\textwidth}{
\small
\textit{Abstract.}
Exact sampling of contingency tables---matrices with prescribed row and column sums---is an important problem in statistics. We present a new algorithm based on the recently introduced \emph{Probabilistic Divide and Conquer} technique. The algorithm improves upon the rejection sampling algorithm for an $m\times n$ contingency table; in particular, it runs in $O(n^{3/2})$ for the well-studied case of a $2\times n$ table under a homogeneity condition on the row sums, which is substantially better than existing Markov Chain Monte Carlo techniques. Unlike MCMC, the runtime depends only on the size of the table and not on the size of the average entry, and the algorithm can be extended to exact sampling of real-valued contingency tables.
}
\end{center}
}

\tableofcontents

\input{ctablesIntro}

\input{ctablesMethod}

\input{ctablesRuntime}

\input{ctablesFixedRows}

\input{ctablesRealBinary}

%\input{ctablesCost}

%\input{ctablesApprox}

\input{ctablesAcknowledgements}

\bibliographystyle{imsart-nameyear}
\bibliography{ctables_ref}

\end{document}

%% file: ctablesIntro.tex
% !TEX root = ctables.tex

\section{Introduction}

\subsection{Background}
Contingency tables are an important data structure in statistics for representing the joint empirical distribution of multivariate data, and are useful for testing properties such as independence between the rows and columns~\cite{GoodCrook} and similarity between two rows or two columns \cite{kitajimamatsui, dyergreenhill}.  %, or (more).

Such statistical tests typically involve defining a test statistic and comparing its observed value to its distribution under the null hypothesis, that all $(r,c)$-contingency tables are equally likely. The null distribution of such a statistic is often impossible to study analytically, but can be approximated by generating contingency tables uniformly at random.
In this paper, we introduce new sampling algorithms for $(r,c)$-contingency tables. 

The most popular approach in the literature for the random generation of contingency tables is Markov Chain Monte Carlo (MCMC) \cite{chen, cryan2003polynomial, cryan2006rapidly, diaconissturmfels, fishman2012counting}, in which one starts with a contingency table and randomly changes a small number of entries in a way that does not affect the row sums and column sums, thereby obtaining a slightly different contingency table. After sufficiently many moves, the new table will be almost independent of the starting table; repeating this process yields \emph{almost} uniform samples from the set of $(r,c)$-contingency tables. The downside to this approach is that the number of steps one needs to wait can be quite large, see for example \cite{bezakova2}, and by the nature of MCMC, one must prescribe this number of steps \emph{before} starting, so the runtime is determined not by the minimum number of steps required but the minimum \emph{provable} number of steps required.

An alternative approach is Sequential Importance Sampling (SIS) \cite{blitzstein2011sequential, chen, chen2006sequential, yoshida2011semigroups}, where one samples from a distribution with computable deviation from uniformity, and weights samples by the inverse of their probability of occurring, to obtain unbiased estimates of any test statistic. Such techniques have proven to be quite fast, but the non-uniformity can present a problem \cite{bezakova2}: depending on the parameters $(r,c)$, the sample can be exponentially far from uniform, and thus the simulated distribution of the test statistic can be very different from the actual distribution despite being unbiased.

There are also extensive results pertaining to counting the number of $(r,c)$-contingency tables, see for example~\cite{Barvinok, BenderTables, Soules, BaldoniSilva, DeLoeraRational, BarvinokInequalities, BarvinokPermanents, GreenhillMcKay, BarvinokIntegerFlows, BarvinokApproximate, BarvinokHartigan, DiaconisGangolli}.  
While our exact sampling algorithms benefit from this analysis, the approximate sampling algorithm in Algorithm~\ref{approximate algorithm} does not require these formulas. 

The special case of $2\times n$ contingency tables has received particular attention in the literature, as it is relatively simple while still being interesting---many statistical applications of contingency tables involve an axis with only two categories (male/female, test/control, etc).  
An asymptotically uniform MCMC algorithm is presented in~\cite{dyergreenhill}. In addition,~\cite{kitajimamatsui} adapted the same chain using coupling from the past to obtain an exactly uniform sampling algorithm.  We describe an exactly uniform sampling algorithm in Section~\ref{main:sect:two} which has an overall lower runtime cost than both approaches and requires no complicated rejection functions nor lookup tables. 

\subsection{Approach}

The main tools we will use in this paper are rejection sampling~\cite{Rejection} and probabilistic divide-and-conquer (PDC)~\cite{PDC}. 

Rejection sampling is a powerful method by which one obtains random variates of a given distribution by sampling from a related distribution, and rejecting observations with probability in proportion to their likelihood of appearing in the target distribution~\cite{Rejection}. 
For many cases of interest, the probability of rejection is in $\{0,1\}$, in which case we sample repeatedly until we obtain a sample that lies inside the target set; this is the idea behind the \emph{exact} Boltzmann sampler~\cite{Boltzmann, Duchon:2011aa}. 

In other applications of rejection sampling, the probability of rejection is in the interval $[0,1]$, and depends on the outcome observed. 
Several examples of this type of rejection sampling are presented in~\cite[Section~3.3]{PDC}; see also~\cite{PDCDSH}. 

PDC is an exact sampling technique which appropriately pieces together samples from conditional distributions. 
The setup is as follows. 
Consider a sample space consisting of a cartesian product $\mathcal{A}\times\mathcal{B}$ of two probability spaces. 
We assume that the set of objects we wish to sample from can be expressed as the collection of pairs $\cL\big((A,B) \,\big|\, (A,B)\in E\big)$ for some $E \subset \mathcal{A}\times \mathcal{B}$, where
\begin{equation}\label{def AB}
  A \in \mathcal{A}, \ B \in \mathcal{B} \ \mbox{ have given distributions},
\end{equation}
\begin{equation}\label{indep AB}
  A , B \ \mbox{ are independent},
\end{equation}
and either 
\begin{enumerate}
	\item[(1)] $E$ is a measurable event of positive probability; or,
	\item[(2)] \begin{enumerate}
			\item[(i)] There is some random variable $T$ on $\mathcal A\times\mathcal B$ which is either discrete or absolutely continuous with bounded density such that $E = \{T=k\}$ for some $k \in \mbox{range}(T)$, and 
			\item[(ii)] For each $a\in\mathcal A$, there is some random variable $T_a$ on $\mathcal B$ which is either discrete or absolutely continuous with bounded density such that $\{b\in\mathcal B:(a,b)\in E\} = \{T_a = k_a\}$ for some $k_a \in \mbox{range}(T_a)$. 
		\end{enumerate}
	\end{enumerate}

We then sample from $\L(\,  (A,B)\, | \,(A,B) \in E)$ in two steps. 
\begin{enumerate}
\item Generate sample from $\L(A\, |\, (A,B)\in E),$ call it $x$.
\item Generate sample from $\L(B\, |\, (x,B) \in E),$ call it $y$.
\end{enumerate}
The PDC lemmas,~\cite[Lemma~2.1]{PDC} in case (1), and~\cite[Lemma~2.1 and Lemma~2.2]{PDCDSH} in case (2), imply that the pair $(x,y)$ is an exact sample from ${\L(\,  (A,B)\, | \,(A,B) \in E)}$. 

In Algorithm~\ref{mainalgorithm}, we utilize the well-known fact that for a geometric random variable $Z$ with parameter $1-q$, the random variable $\eta(q) := \mathbbm{1}(Z\textrm{ is odd})$ is Bernoulli with parameter $\frac{q}{1+q}$, which is \emph{independent} of $(Z-\eta)/2$, which is geometrically distributed with parameter $1-q^2$.
In our first case of interest, nonnegative integer-valued contingency tables, $q$ is of the form $q_j = c_j / (m+c_j)$, where $c_j$ is the $j$-th column sum, $j=1,2,\ldots, n$; see lemmas~\ref{lemma:uniform} and~\ref{lemma:q}. 

The use of rejection sampling is optimized by picking a distribution which is close to the target distribution. 
The random variable $\eta(q_j)$ serves as an a priori estimate for the true distribution of the least significant bit of an entry under the conditional distribution. 
Decomposing a geometric random variable by its bits appears to be a relatively recent technique for random sampling. 
It was effectively utilized in~\cite{PDC} for the random sampling of integer partitions of a fixed size~$n$, i.e., an exact Boltzmann sampler, with $O(1)$ expected number of rejections. 

After the least significant bit of each entry of the table is sampled, we return to the same sampling problem with reduced row sums and column sums. 
More importantly, after each step, we choose new parameters $q_j$, $j=1,2,\ldots,n$, based on these new column sums, which tilts the distribution more in favor of this new table. 
If we had sampled the geometric random variables directly, it would be equivalent to choosing one value of $q$ for the entire algorithm, and sampling from Bernoulli random variables with parameters $q/(1+q)$, $q^2/(1+q^2)$, $q^4/(1+q^4)$, etc. 
Using this new approach, we are sampling from Bernoulli random variables with parameters $q/(1+q)$, $q'/(1+q')$, $q''/(1+q'')$, etc., where $q'$, $q''$, etc., are chosen after each iteration, and more effectively target the current set of row sums and column sums.

\begin{remark}\label{quasi:remark}
One might define a new type of random variable, say, a \emph{quasi}-geometric random variable, denoted by $Q$, which is defined as
\[ Q(q,q',q'',\ldots) = \Bern\left( \frac{q}{1+q}\right) + 2\Bern\left(\frac{q'}{1+q'}\right) + 4 \Bern\left(\frac{q''}{1+q''}\right) + \ldots, \]
where all of the Bernoulli random variables are mutually independent. 
In our case, we choose $q'$ after an iteration of the algorithm completes with $q$. 
Subsequently, we use $q''$ after an iteration of the algorithm completes with $q'$, etc.
This quasi-geometric random variable has more degrees of freedom than the usual geometric random variable, although we certainly lose other nice properties which are unique to geometric random variables. 
\end{remark}

\begin{remark}\label{bit:remark}
In~\cite{Barvinok}, a question is posed to obtain the asymptotic joint distribution of a small subset of entries in a random contingency table, and an example of a specific table is presented which shows that using independent geometric random variables as the marginal distributions of small subsets of this joint density leads to joint distributions which are not indicative of the uniform measure over the set of such contingency tables. 
We surmise it may be more fruitful to consider the joint distribution of the $r$ least significant bits of the usual geometric distribution; or, alternatively, to consider a quasi-geometric random variable defined in Remark~\ref{quasi:remark} with given parameters $q$, $q'$, $q''$, etc.
\end{remark}

In addition to sampling from nonnegative \emph{integer}-valued tables, one may also wish to sample from nonnegative \emph{real}--valued tables with real--valued row sums and column sums. 
For a real-valued contingency table, the role of geometric random variables is replaced with exponential random variables. 
Instead of sampling from the smallest bit of each entry, we instead sample the fractional parts first, and what remains is an integer--valued table with integer--valued row sums and column sums. 

To obtain a good candidate distribution for the fractional part of a given entry, we utilize two well-known facts summarized in the lemma below.  
For real $x$, $\{x \}$ denotes the fractional part of $x$, and $\lfloor x \rfloor $ denotes the integer part of $x$, so that $x = \lfloor x \rfloor + \{x \}$. 

\begin{lemma}
Let $Y$ be an exponentially distributed random variable with parameter $\lambda>0$, then:
\begin{itemize}
\item the integer part, $\lfloor Y \rfloor$, and the fractional part, $\{Y\}$, are independent~\cite{Rejection, SteutelThiemann};
\item $\lfloor Y\rfloor$ is geometrically distributed with parameter $1-e^{-\lambda}$, and $\{Y\}$ has density $f_\lambda(x) = \lambda e^{-\lambda x} / (1-e^{-\lambda})$, $0\leq x < 1$. 
\end{itemize}
\end{lemma}

In this case, the recommended approach is the following: 

\begin{enumerate}
\item Sample the fractional part of the entries of the table first, according to the conditional distribution; 
\item Sample the remaining integer part of the table. 
\end{enumerate}

If an exact sample of each of the items above can be obtained, then an application of PDC implies that the sum of the entries of the two tables has the uniform distribution over nonnegative real--valued tables. 

For $2 \times n$ tables, Algorithm~\ref{PDC DSH 2 by n} is particularly elegant, as it is an exact sampling algorithm which does not require the computation of \emph{any} rejection functions nor lookup tables, and instead exploits a property of conditional distributions. 
We also present several alternative exact sampling PDC algorithms for tables with entries with given marginal distributions, subject to relatively minor conditions.  
 
Finally, by dropping the requirement that an algorithm should produce a sample uniformly at random from a given set of contingency tables, we provide an approximate sampling algorithm in Algorithm~\ref{approximate algorithm} for which any given bit is in its incorrect proportion by at most $m+2$, where $m$ is the number of rows; see Lemma~\ref{runtime:lemma}.
An approximate sampling algorithm for the fractional part of the entries of a given table is given in Algorithm~\ref{alternative fractional real}.

The paper is organized as follows. 
Section~\ref{sect:main} contains the main algorithms for the random sampling of $(r,c)$-contingency tables, including an approximate sampling algorithm which does not rely on any enumeration formulas. 
Section~\ref{sect:method} contains relevant properties of contingency tables. 
Section~\ref{sect:proofs} contains the proof that Algorithm~\ref{mainalgorithm} is uniform over all tables. %, and gives run-time estimates. 
Section~\ref{sect:two} discusses the PDC algorithm in the special case when there are exactly $2$ rows. 
Section~\ref{sect:other} formulates similar PDC algorithms for a joint distribution with given marginal probability distributions under mild restrictions.

\section{Main Algorithms}\label{sect:main}

\subsection{Integer-valued tables}
\label{int:tables}
Our first algorithm is Algorithm~\ref{mainalgorithm} below, which uniformly samples from the set of nonnegative integer-valued contingency tables with any given size and row sums and column sums. 
The algorithm is recursive (though our implementation is iterative). 
Each step of the algorithm samples the least significant bit of a single entry in the table, \emph{in proportion to its prevalence under the conditional distribution}, via rejection sampling, and once the least significant bit of all entries in the table have been sampled, all of the row sums and column sums are reduced by at least a factor of two. 
The algorithm repeats with the next significant bit, etc., until all remaining row sums and column sums are 0. 

For a given set of row sums $r = (r_1, \ldots, r_m)$ and column sums $c = (c_1, \ldots, c_n)$, let $\Sigma(r,c)$ denote the number of $(r,c)$-contingency tables. 
Let $\mathcal{O}$ denote an $m\times n$ matrix with entries in $\{0,1\}$.  
For any set of row sums and column sums $(r,c)$, let $\Sigma(r,c,\mathcal{O})$ denote the number of $(r,c)$-contingency tables with entry $(i,j)$ forced to be even if the $(i,j)$th entry of $\mathcal{O}$ is 1, and no restriction otherwise. 
Let $\mathcal{O}_{i,j}$ denote the matrix which has entries with value $1$ in the first $j-1$ columns, and entries with value $1$ in the first $i$ rows of column $j$, and entries with value 0 otherwise. 

Define for $1 \leq i \leq m-2, 1 \leq j \leq n-2$, 
\begin{equation}\label{rejection:ij}\small
 f(i,j,k,r,c) :=  \frac{\Sigma\left( \begin{array}{l}(\ldots,r_{i}-k,\ldots), \\ (\ldots,c_{j}-k,\ldots), \\ \mathcal{O}_{i,j}\end{array} \right)}{\Sigma\left( \begin{array}{l}(\ldots,r_{i}-1,\ldots), \\ (\ldots,c_{j}-1,\ldots), \\ \mathcal{O}_{i,j}\end{array} \right) + \Sigma\left( \begin{array}{l}(\ldots,r_{i},\ldots), \\ (\ldots,c_{j},\ldots), \\ \mathcal{O}_{i,j}\end{array} \right)}.
\end{equation}

When $j = n-1$, we have a slightly different expression.  
Let $q_j := \frac{c_j}{m+c_j}$, and 
let $y_j := q_j^{-1} = 1+\frac{m}{c_j}$.  
Let $b(k)$ be such that $r_i-k-b(k)$ is even.  Then we define for $1 \leq i \leq m-2,$
\begin{equation} \label{equation:in} f(i,n-1,k,r,c) := \end{equation}
\begin{equation*}\small \frac{\Sigma\left( \begin{array}{l}(\ldots,r_{i}-k-b(k),\ldots), \\ (\ldots,c_{n-1}-k,c_n-b(k)), \\ \mathcal{O}_{i,n}\end{array} \right)\,\cdot\, y_n^{b(k)}}{\Sigma\left( \begin{array}{l}(\ldots,r_{i}-1-b(1),\ldots), \\ (\ldots,c_{n-1}-1,c_n-b(1)), \\ \mathcal{O}_{i,n}\end{array} \right)\,\cdot\, y_n^{b(1)}+\Sigma\left( \begin{array}{l}(\ldots,r_{i}-b(0),\ldots), \\ (\ldots,c_{n-1},c_n-b(0)), \\ \mathcal{O}_{i,n}\end{array} \right)\,\cdot\, y_n^{b(0)}}. 
\end{equation*}
When $i=m-1$, let $v(k)$ be such that $c_j - k-v(k)$ is even.  Then we define for $1 \leq j \leq n-2,$
\begin{equation}\small \label{rejection:mj}
f(m-1,j,k,r,c) := 
\end{equation}
\begin{equation*}
 \frac{\Sigma\left( \begin{array}{l} (\ldots,r_{m-1}-k,r_m-v(k)), \\ (\ldots,c_{j}-k-v(k),\ldots), \\ \mathcal{O}_{m,j}\end{array}\right)\,\cdot\, y_{j}^{v(k)}}{\Sigma\left( \begin{array}{l} (\ldots,r_{m-1}-1,r_m-v(1)), \\ (\ldots,c_{j}-1-v(1),\ldots), \\ \mathcal{O}_{m,j}\end{array}\right)\,\cdot\, y_{j}^{v(1)}+\Sigma\left( \begin{array}{l} (\ldots,r_{m-1},r_m-v(0)), \\ (\ldots,c_{j}-v(0),\ldots), \\ \mathcal{O}_{m,j}\end{array}\right)\,\cdot\, y_{j}^{v(0)}}. 
\end{equation*}
Finally, when $i=m-1$ and $j=n-1$, let $v(k)$ be such that $c_{n-1} - k-v(k)$ is even, let $\gamma(k)$ be such that $r_{m-1}-k-\gamma(k)$ is even, and let $b(k)$ be such that $c_n-\gamma-b(k)$ is even.  Then we define
\begin{equation*}
A \equiv \Sigma\left(\begin{array}{l}(\ldots,r_{m-1}-k-\gamma(k),r_m-v(k)-b(k)), \\ (\ldots,c_{n-1}-k-v(k),c_n-\gamma(k)-b(k)), \\ \mathcal{O}_{m,n}\end{array} \right)\,\cdot \, y_{n-1}^{v(k)}\, y_n^{\gamma(k) + b(k)},
\end{equation*}
\begin{equation*}
B \equiv \Sigma\left(\begin{array}{l}(\ldots,r_{m-1}-1-\gamma(1),r_m-v(1)-b(1)), \\ (\ldots,c_{n-1}-k-v(1),c_n-\gamma(1)-b(1)), \\ \mathcal{O}_{m,n}\end{array} \right)\,\cdot \, y_{n-1}^{v(1)}\, y_n^{\gamma(1) + b(1)},
\end{equation*}
\begin{equation*}
C \equiv \Sigma\left(\begin{array}{l}(\ldots,r_{m-1}-\gamma(0),r_m-v(0)-b(0)), \\ (\ldots,c_{n-1}-v(0),c_n-\gamma(0)-b(0)), \\ \mathcal{O}_{m,n}\end{array} \right)\,\cdot \, y_{n-1}^{v(0)}\, y_n^{\gamma(0) + b(0)},
\end{equation*}
\begin{equation}\label{equation:mn} f(m-1,n-1,k,r,c) :=  \frac{A}{B+C}. \end{equation}

The algorithm itself is then simple and straightforward. 

\begin{algorithm}[H]
\caption{Generation of uniformly random $(r,c)$-contingency table}
\label{mainalgorithm}
\begin{algorithmic}[1]
\STATE $M \leftarrow \max( \max_i r_i, \max_j c_j)$.
\STATE $t \leftarrow$ \mbox{$m \times n$ table with all 0 entries}.
\FOR {$b = 0,1,\ldots, \lceil\log_2(M)\rceil$}

   \STATE Let $\sigma_R$ denote any permutation such that $\sigma_R \circ r$ is in increasing order.
   \STATE Let $\sigma_C$ denote any permutation such that $\sigma_C \circ c$ is in increasing order.

   \STATE $r \leftarrow \sigma_R \circ r$.   
   \STATE $c \leftarrow \sigma_C \circ c$.

  \FOR {$j=1,\ldots,n-1$}\label{column:loop}

  \FOR {$i=1,\ldots,m-1$}\label{row:loop}
  \IF{$U < f(i,j,0,r,c)$} \label{reject:main}
     \STATE $\epsilon_{i,j} \leftarrow 0$
  \ELSE
     \STATE $\epsilon_{i,j} \leftarrow 1$
     \ENDIF 
  \STATE $r_i \leftarrow r_i - \epsilon_{i,j}$.      
  \STATE $c_j \leftarrow c_j - \epsilon_{i,j}$.      
  \ENDFOR \label{column:loop:end}
  
  \STATE $\epsilon_{m,j} \leftarrow c_j  \mod 2$ \label{line:last:column}
  \STATE $c_j \leftarrow c_j - \epsilon_{m,j}$.
  \STATE $c_j \leftarrow c_j / 2 $
  
  \STATE $r_m \leftarrow r_m - \epsilon_{m,j}$.

  \ENDFOR \label{row:loop:end}
     \FOR{$i=1,\ldots,m$} \label{line:end:of:third}
         \STATE $\epsilon_{i,n} \leftarrow r_i \mod 2$ \label{line:last:column}
         \STATE $r_i \leftarrow r_i - \epsilon_{i,n}$
         \STATE $r_i \leftarrow r_i / 2$
         \STATE $c_n \leftarrow c_n - \epsilon_{i,n}$
      \ENDFOR\label{line:last:three}
      
       \STATE $c_n \leftarrow  c_n / 2 $. 

	\STATE $\epsilon \leftarrow \sigma_R^{-1}\circ {\bf \epsilon}$  (Apply permutation to rows)
	\STATE $\epsilon \leftarrow \sigma_C^{-1}\circ {\bf \epsilon}$  (Apply permutation to columns)

	\STATE $t \leftarrow t+2^b \epsilon$. \label{line:combine}
  \ENDFOR
\end{algorithmic}
\end{algorithm}

To summarize Algorithm~\ref{mainalgorithm}: sample the least significant bit of each entry in the table by its proportion under the conditional distribution, one by one, and combine them with the previously sampled bits using probabilistic divide-and-conquer (PDC)~\cite{PDC}, until all entries of the table have had their least significant bit sampled.  The next step is the observation that, conditional on having observed the least significant bits of the table, the remaining part of each entry is even, with even row sums and column sums, and so we can divide each of the entries by two, as well as the row sums and column sums, and we have a repeat of the same problem with reduced row sums and column sums. 
The final step at each iteration is an application of PDC, which implies that the current iteration can be combined with previous iterations as we have done in Line~\ref{line:combine}. 

\begin{remark}
The cost to evaluate $f$ at each iteration, or more precisely, \emph{the cost to decide Line~\ref{reject:main}}, is currently the main cost of the algorithm, which requires on average the leading two bits of the evaluation of $f$, see~\cite{KnuthYao}.
Fortunately, we do not need to evaluate the quantities exactly, and in fact typically we just need a few of the most significant bits.  
\end{remark}

\begin{remark}\label{Igor:remark}
Note that the random sampling of the least significant bits of the table at each iteration is \emph{not} equivalent to the random sampling of binary contingency tables. 
The task is considerably easier in general, since we do not have to obtain a fixed target at each iteration. 
\end{remark}

\begin{theorem}\label{main theorem}
Algorithm~\ref{mainalgorithm} 
requires an expected $O(m\, n\, \log(M))$ random bits, where $M$ is the largest row sum or column sum.
\end{theorem}

\begin{remark}{\rm
Theorem~\ref{main theorem} is a statement about the inherent amount of randomness in the sampling algorithm, and \emph{not} a complete runtime cost of Algorithm~\ref{mainalgorithm}, which requires the computation of non-random quantities for which at present no polynomial-time algorithm is known. 
We note, however, that it is universal with respect to all possible row sums and column sums; in other words, if you could count, Algorithm~\ref{mainalgorithm} is an explicit, asymptotically efficient sampling algorithm. 
}\end{remark}

\begin{remark}{\rm
Assuming Conjecture~\ref{main:conjecture} below, Algorithm~\ref{mainalgorithm} is a simple, explicit polynomial-time algorithm for the uniform generation of $m\times n$ $(r,c)$-contingency tables.
It samples the least significant bit of each entry of the table one at a time, which is where the factor of $O(m\, n\, \log(M))$ is derived.  
It does \emph{not} require the exact calculation for the number of tables in advance, but rather requires that the leading $r$ bits of related tables can be computed exactly in time $o(2^r)$ times a polynomial in $m$ and $n$ and logarithmically in $M$, where $M$ is the largest row sum or column sum.  
}\end{remark}

\begin{conjecture}\label{main:conjecture}
Let $\mathcal{O}$ denote an $m\times n$ matrix with entries in $\{0,1\}$.  
For any set of row sums and column sums $(r,c)$, let $\Sigma(r,c,\mathcal{O})$ denote the number of $(r,c)$-contingency tables with entry $(i,j)$ forced to be even if the $(i,j)$th entry of $\mathcal{O}$ is 1, and no restriction otherwise. 
Let $M$ denote the largest row sum or column sum. 
Then for any $\mathcal{O}$, the leading $r$ bits of $\Sigma(r,c,\mathcal{O})$ can be computed in $o(2^r m^\alpha n^\beta \log^\gamma(M))$ time, for some absolute constants $\alpha, \beta, \gamma \geq 0$. 
\end{conjecture}

\begin{remark}{\rm
An asymptotic formula for a large class of $(r,c)$-contingency tables is given in~\cite{BarvinokHartigan} which can be computed in time $O(m^2n^2)$, but it is unclear whether the bounds can be made quantitative, or whether the approach can be extended to provide an asymptotic expansion.  An asymptotic expansion is given in~\cite[Theorem 5]{CanfieldMcKay} for binary contingency tables with equal row sums and equal column sums which is computable in polynomial time, although because it is based on Edgeworth expansions it is not guaranteed to converge as more and more terms are summed. 
}\end{remark}

An equivalent formulation of the rejection function in Algorithm~\ref{mainalgorithm} is given by a joint distribution of random variables, which we now describe; see also~\cite{BarvinokHartigan}.

\begin{tabular}{p{1.6cm}p{10cm}}
$\Geo(q)$ & Geometric distribution with probability of success $1-q$, for $0<q<1$, with $\Pr\left(\Geo(q)=k\right)=(1-q)q^k$, $k=0,1,2,\ldots\, $. \\[.15cm]
$\NB(m, q)$ & Negative binomial distribution with parameters $m$ and $1-q$, given by the sum of $m$ independent $\Geo(q)$ random variables, with 
\[\Pr(\NB(m,q) = k) = \binom{m+k-1}{k}(1-q)^m q^k.\] \\[.15cm]
U % non-italicized to distinguish distribution from random variable
& Uniform distribution on $[0,1]$. We will also use $U$ in the context of a random variable; $U$ should be considered independent of all other random variables, including other instances of $U$. \\[.15cm]
$\Bern(p)$ & Bernoulli distribution with probability of success $p$. Similarly to $U$, we will also use it as a random variable. \\[.15cm]
$\xi_{i,j}(q)$ & $\Geo(q)$ random variables which are independent for distinct pairs $(i,j)$, $1\le i\le m$, $1\le j\le n$. \\[.15cm]
$\xi_{i,j}'(q,c_j)$ & Random variables which have distribution \[\L\left(\xi_{i,j}(q) \bigg| \sum_{\ell=1}^m \xi_{\ell,j}(q) = c_j\right),\] and are independent of all other random variables $\xi_{i,\ell}(q)$ for $\ell \ne j$. \\[.15cm]
$2\xi_{i,j}''(q,c_j)$ & Random variables which have distribution \[\L\left(2\xi_{i,j}(q^2) \bigg| \sum_{\ell=1}^m 2\xi_{\ell,j}(q^2) = c_j \right)\] 
and are independent of all other random variables $\xi_{i,\ell}(q)$ for $\ell \ne j$. \\[.15cm]
$\eta'_{i,j,s}(q, c_j)$ & Random variables which have distribution 
\[ \L\left( \xi_{i,j}(q) \middle| \sum_{\ell=1}^s 2\xi_{\ell,j}(q^2) + \sum_{\ell=s+1}^m \xi_{\ell,j}(q) = c_j  \right), \]
and are independent of all other random variables $\xi_{i,\ell}(q)$ for $\ell \ne j$. \\[.15cm]
$2\eta''_{i,j,s}(q, c_j)$ & \ \ Random variables which have distribution 
\[ \L\left( 2\xi_{i,j}(q^2) \middle| \sum_{\ell=1}^s 2\xi_{\ell,j}(q^2) + \sum_{\ell=s+1}^m \xi_{\ell,j}(q) = c_j \right), \]
and are independent of all other random variables $\xi_{i,\ell}(q)$ for $\ell \ne j$. \\[.15cm]
$\mathbf q$ & The vector $(q_1, \ldots, q_n)$, where $0<q_i<1$ for all $i=1,\ldots,n$. \\[.15cm]
$R_i$ & = $(\xi_{i,1}, \xi_{i,2}, \ldots, \xi_{i,n})$ for $i=1,2,\ldots,m$. \\[.15cm]
$C_j$ & = $(\xi_{1,j}, \xi_{2,j}, \ldots, \xi_{m,j})$ for $j=1,2,\ldots,n$. \\[.15cm]
\end{tabular}
 
The rejection probability is then proportional to 
\begin{align*}
\frac{ \P(\epsilon_{1,1} = k | E)}{\P(\epsilon_{1,1} = k)} & = \P(E | \epsilon_{1,1}=k) \\
& \propto 
\P\left( \begin{array}{ll} 
2\, \xi_{i,1}(q_1^2) + \sum_{i=2}^m \xi_{i,1}(q_1) = c_1 - k,& 2\xi_{1,1}(q_1^2) + \sum_{j=2}^{n} \xi_{1,j}(q_j) = r_1- k,\\
C_2 = c_2, & R_2 = r_2\\
\qquad \vdots & \qquad \qquad \qquad \qquad \vdots \\
C_n = c_n, & R_m = r_m
\end{array}\right)  \\
& \propto
 \P\left( \begin{array}{l}  
2\eta'_{1,1,1}(q_1,c_1) + \sum_{j=2}^{n} \xi'_{1,j}(q_1,c_1) = r_1- k \\
R_2=r_2 \\ \vdots \\ R_m = r_m \end{array}
\middle| \begin{array}{l} C_1 = c_1 \\ \epsilon_{1,1} = k \\ C_2 = c_2 \\ \vdots \\ C_n = c_n  \end{array} \right) \\
& \ \ \ \  \times  \P\left( 2\, \xi_{i,1}(q_1^2) + \sum_{i=2}^m \xi_{i,1}(q_1) = c_1 - k \middle| \begin{array}{l} \epsilon_{1,1} = k \\ C_2 = c_2  \\ \vdots \\ C_n = c_n\end{array} \right). \\
& \propto
 \P\left( \begin{array}{l}  
2\eta'_{1,1,1}(q_1,c_1) + \sum_{j=2}^{n} \xi'_{1,j}(q_1,c_1) = r_1- k \\
R_2=r_2 \\ \vdots \\ R_m = r_m \end{array}
\middle| \begin{array}{l} C_1 = c_1 \\ \epsilon_{1,1} = k \\ C_2 = c_2 \\ \vdots \\ C_n = c_n  \end{array} \right) \\
& \ \ \ \  \times  \P\left( 2\, \xi_{i,1}(q_1^2) + \sum_{i=2}^m \xi_{i,1}(q_1) = c_1 - k \right). \\
\end{align*}

The last expression can be substituted for the rejection function $f$ when $i=j=1$, i.e., we have just shown that
\ignore{
 \begin{align}\label{f11}
 f(1,1,k, r, c) \propto \ & 
 \P\left( \begin{array}{l}  
2\xi'''_{1,1}(q_1^2) + \sum_{j=2}^{n} \xi'_{1,j}(q_1) = r_1- k \\
R_2'=r_2 \\ \vdots \\ R_m' = r_m \end{array}
\middle| \begin{array}{l} C_1 = c_1 \\ \epsilon_{1,1} = k \\ C_2 = c_2 \\ \vdots \\ C_n = c_n  \end{array} \right) \\
\nonumber & \ \ \ \  \times  \P\left( 2\, \xi_{i,1}(q_1^2) + \sum_{i=2}^m \xi_{i,1}(q_1) = c_1 - k \right). \\
\end{align}}
 \begin{align}\label{f11}
 f(1,1,k, r, c)  \propto\ &
 \P\left( \begin{array}{lll}  
\eta_{1,1,1}'(q_1, c_1) &+  \sum_{\ell=2}^{n} \xi'_{1,\ell}(q_\ell, c_\ell) &= r_1-k \\
\eta_{2,1,1}''(q_1, c_1) &+  \sum_{\ell=2}^{n} \xi'_{2,\ell}(q_\ell, c_\ell) &= r_2 \\
\qquad \vdots \\
\eta_{m,1,1}''(q_1, c_1) &+  \sum_{\ell=2}^{n} \xi'_{m,\ell}(q_\ell, c_\ell) &= r_{m} \\
\end{array}\right) \\
\nonumber & \ \ \ \  \times  \P\left( 2\, \xi_{i,j}(q_j^2) + \sum_{i=2}^m \xi_{i,j}(q_j) = c_j - k \right). 
\end{align} 

Algorithm~\ref{mainalgorithm} samples $\L(\epsilon_{1,1} | E)$, $\L(\epsilon_{2,1} | E, \epsilon_{1,1})$, $\L(\epsilon_{3,1} | E,\epsilon_{1,1},\epsilon_{2,1})$, etc., until the entire first column is sampled.  Then it starts with the top entry of the second column and samples the least significant bits from top to bottom according to the conditional distribution.  
Generalizing equation~\eqref{f11}, we have

 \begin{align}\label{fij}
\hspace{-0.25in} f(i,j,k, r, c)  \propto \ &
 \P\left( \begin{array}{llll}  
\sum_{\ell=1}^{j-1}2\xi''_{1,\ell}(q_\ell^2, c_\ell) &+ \eta_{1,j,i}'(q_j, c_j) &+  \sum_{\ell=j+1}^{n} \xi'_{1,j}(q_\ell, c_\ell) &= r_1 \\
\sum_{\ell=1}^{j-1}2\xi''_{2,\ell}(q_\ell^2, c_\ell) &+ \eta_{2,j,i}'(q_j, c_j) &+  \sum_{\ell=j+1}^{n} \xi'_{2,j}(q_\ell, c_\ell) &= r_2 \\
\qquad \vdots \\
\sum_{\ell=1}^{j-1}2\xi''_{i-1,\ell}(q_\ell^2, c_\ell) &+ \eta_{i-1,j,i}'(q_j, c_j) &+  \sum_{\ell=j+1}^{n} \xi'_{i-1,j}(q_\ell, c_\ell) &= r_{i-1} \\
\sum_{\ell=1}^{j-1}2\xi''_{i,\ell}(q_\ell^2, c_\ell) &+ \eta_{i,j,i}''(q_j, c_j) &+  \sum_{\ell=j+1}^{n} \xi'_{i,j}(q_\ell, c_\ell) &= r_{i}-k \\
\sum_{\ell=1}^{j-1}2\xi''_{i+1,\ell}(q_\ell^2, c_\ell) &+ \eta_{i+1,j,i}''(q_j, c_j) &+  \sum_{\ell=j+1}^{n} \xi'_{i+1,j}(q_\ell, c_\ell) &= r_{i+1} \\
\qquad \vdots \\
\sum_{\ell=1}^{j-1}2\xi''_{m,\ell}(q_\ell^2, c_\ell) &+ \eta_{m,j,i}''(q_j, c_j) &+  \sum_{\ell=j+1}^{n} \xi'_{m,j}(q_\ell, c_\ell) &= r_{m} \\
\end{array}\right) \\
\nonumber & \ \ \ \  \times  \P\left( \sum_{\ell = 1}^{i} 2\, \xi_{\ell,j}(q_j^2) + \sum_{\ell=i+1}^m \xi_{\ell,j}(q_\ell) = c_j - k \right). 
\end{align} 

Note that $\sum_{\ell=1}^i \xi_{\ell,j}(q_j^2)$ is the sum of $i$ i.i.d.~geometric random variables, and is hence a negative binomial distribution with parameters $i$ and $q_j^2$.  
We thus have the sum of two independent negative binomial distributions:
\[ \P\left( \sum_{\ell=1}^i 2\xi_{\ell,j}(q_j^2) + \sum_{\ell=i+1}^{m} \xi_{\ell,j}(q_j) = c_j- k\right) = \P\left( 2\NB(i,q_j^2) + \NB(m-i,q_j) = c_j - k\right). \]

We note that each random variable in the probability above has an explicitly computable and simple probability mass function, which we write below.
\begin{align}
\nonumber \Pr\left( \xi_{i,j}'(q,c_j) = k\right) & = \Pr\left(\xi_{i,j}(q)=k\right) \frac{\NB(m-1,q)\{c_j-k\}}{\NB(m,q)\{c_j\}} \\
\label{mass:1} & = \frac{\binom{(m-1)+(c_j-k)-1}{c_j-k}}{\binom{m+c_j-1}{c_j}}, \qquad k=0,1,\ldots,c_j.
 \end{align}
Denote by $\epsilon_{i,j}$ the observed value of the parity bit of $\xi_{i,j}$, $i=1,2,\ldots,m$, $j=1,2,\ldots,n$. Let $c_j' = \frac{c_j - \sum_{i=1}^m \epsilon_{i,j}}{2}.$  We have
\begin{align}
\nonumber \Pr\left( 2\xi_{i,j}''(q,c_j') = k\right) & = \Pr\left(2\xi_{i,j}(q^2)=k\right) \frac{\NB(m-1,q^2)\left\{c_j' - \frac{k}{2}\right\}}{\NB(m,q^2)\left\{c_j'\right\}} \\
\label{mass:2} & = \frac{\binom{(m-1)+\left(c_j'-\frac{k}{2}\right)-1}{c_j'-\frac{k}{2}}}{\binom{m+c_j'-1}{c_j'}}, \qquad k=0,2,4,\ldots, 2c_j'. 
\end{align}
Let $c_j'' = c_j - \sum_{\ell=1}^s \epsilon_{\ell,j}$, $j=1,2,\ldots,n$.  We have
\begin{equation}\label{mass:3}
 \P(2\eta''_{i,j,s}(q,c_j'') = k) = \P\left(2\xi_{i,j}(q^2) = k)\right) \  \frac{\P(2\NB(s-1,q^2) + \NB(m-s,q) = c_j'' -  k) }{\P(2\NB(s,q^2) + \NB(m-s,q) = c_j'')}, 
\end{equation}
for $k = 0, 2, \ldots, 2 \bigl\lfloor \frac{c_j''}{2}\bigr\rfloor.$
\begin{equation}\label{mass:3}
 \P(\eta'_{i,j,s}(q,c_j'') = k) = \P\left(\xi_{i,j}(q) = k)\right) \  \frac{\P(2\NB(s,q^2) + \NB(m-s-1,q) = c_j'' -  k) }{\P(2\NB(s,q^2) + \NB(m-s,q) = c_j'')}, 
\end{equation}
for $k = 0, 1,\ldots, c_j''.$

\subsection{Approximate sampling of integer-valued tables}
\label{sect:approx}

If an approximate sampling algorithm for $(r,c)$-contingency tables is deemed acceptable, Lemma~\ref{runtime:lemma} below demonstrates that the least significant bit of entry $(i,j)$ is closely distributed as $\L\left( \Bern\left(\frac{q_j}{1+q_j}\right)\right)$, with $q_j = \frac{c_j}{m+c_j}$, and so one could sample the least significant bits of the table one at a time according to this \emph{unconditional} distribution, without applying any other rejections other than the parity constraints, and then apply the recursion until all row sums and column sums are 0.
The following lemma shows that there is a quantitative limit to the bias introduced by this and related approaches. 

\ignore{
The overall runtime of Algorithm~\ref{mainalgorithm} is difficult to cost directly, owing to its recursive nature and dependence on the row sums and column sums, which change unpredictably at each iteration. 
Fortunately, since the rejection step is only over two states, one of them is necessarily accepted with probability 1, and we can bound the total runtime of a given rejection step by the expected number of rejections before we observe the outcome which is accepted with probability 1. 
Our estimate for total runtime cost in Lemma~\ref{runtime:lemma} below is therefore pessimistic, and can possibly be improved by a more careful analysis. 
}
\begin{lemma}\label{runtime:lemma}
Any algorithm which uses $\L\left( \Bern\left(\frac{q_j}{1+q_j}\right)\right),$ where $q_j = \frac{c_j}{m+c_j}$, as the surrogate distribution for $\L(\epsilon_{i,j} | E)$ in rejection sampling, assuming each outcome in $\{0,1\}$ has a \emph{positive} probability of occurring, accepts a bit with an incorrect proportion bounded by at most $m+2$, where $m$ is the number of rows. 
\end{lemma}
\begin{proof}
We consider the worst case example, which is mild since there are only two possible outcomes for each entry in each iteration of the algorithm.

Since we are normalizing by the max over the two states, at least one of the states is accepted with probability 1, and so the total number of rejections is bounded from above by the wait time until this state is generated. 
Since the random variable generating the bit is Bernoulli with parameter $\frac{q_j}{1+q_j}$, we have
\[ \Pr\left( \Bern\left(\frac{q_j}{1+q_j}\right) = 0\right) =  \frac{m+c_j}{m+2c_j} \geq \frac{1}{2}. \]
\[ \Pr\left( \Bern\left(\frac{q_j}{1+q_j}\right) = 1\right) =  \frac{c_j}{m+2c_j} \geq \frac{1}{m+2}. \]
Thus, at worst we accept a bit with proportion $m+2$ times more likely than the other state. 
\end{proof}

We suggest a slightly more involved approach, which is to treat the row sum conditions as essentially independent, and reject each bit generated as if it was independent of all other rows. 
Algorithm~\ref{approximate algorithm} below uses Equation~\eqref{fij} and assumes that the rows are independent, which greatly simplifies the rejection probability formula. 
A rejection function is then given by 
\begin{align*}
& F(i,j,m,n,{\bf q}, r, c,k,{\bf \epsilon}) \\
 & := \Pr\left( \sum_{\ell=1}^{j-1} 2\xi''_{i,\ell}(q_\ell^2, c_j, \epsilon) + 2\eta''_{i,j,i} + \sum_{\ell=j+1}^{n} \xi'_{i,\ell}(q_\ell) = r_i - k\right) \\
 & \ \ \ \times \Pr\left( \sum_{\ell=1}^i 2\xi_{\ell,j}(q_j^2) + \sum_{\ell=i+1}^{m} \xi_{\ell,j}(q_j) = c_j - k\right), \qquad k \in \{0,1\}. \\
 \end{align*}
Note that the first term is a probability over a sum of \emph{independent} random variables, and as such can be computed using convolutions in time $O(M^2)$ or fast Fourier transforms in time $O(n\, M \log M)$. 
Further speedups are possible since only the first few bits of the function are needed on average. 

\begin{algorithm}[H]
\caption{Generation of approximately-uniformly random $(r,c)$-contingency table}
\label{approximate algorithm}
\begin{algorithmic}[1]
\STATE $M \leftarrow \max( \max_i r_i, \max_j c_j)$.
\STATE $t \leftarrow$ \mbox{$m \times n$ table with all 0 entries}.
\FOR {$b = 0,1,\ldots, \lceil\log_2(M)\rceil$}

   \STATE Let $\sigma_R$ denote any permutation such that $\sigma_R \circ r$ is in increasing order.
   \STATE Let $\sigma_C$ denote any permutation such that $\sigma_C \circ c$ is in increasing order.

   \STATE $r \leftarrow \sigma_R \circ r$.   
   \STATE $c \leftarrow \sigma_C \circ c$.

  \FOR {$j=1,\ldots,n-1$}\label{column:loop}

  \FOR {$i=1,\ldots,m-1$}\label{row:loop}
  \STATE $q_j \leftarrow c_j/(m+c_j).$ 
  \STATE $\epsilon_{i,j} \leftarrow \Bern(q_j/(1+q_j))$. \label{line:start-of-loop}

  \IF{$U > \frac{F(i,j,m,n,{\bf q}, r_i,c_j,\epsilon_{i,j})}{\max_{\ell\in\{0,1\}} F(i,j,m,n,{\bf q}, r_i,c_j,\ell)}$} \label{reject}
     \STATE \textbf{goto} Line~\ref{line:start-of-loop}.
     \ENDIF 
  \STATE $r_i \leftarrow r_i - \epsilon_{i,j}$.      
  \STATE $c_j \leftarrow c_j - \epsilon_{i,j}$.     
  \ENDFOR \label{column:loop:end}
  
  \STATE $\epsilon_{m,j} \leftarrow c_j  \mod 2$ \label{line:last:column}
  \STATE $c_j \leftarrow c_j - \epsilon_{m,j}$.
  \STATE $c_j \leftarrow c_j / 2 $
  
  \STATE $r_m \leftarrow r_m - \epsilon_{m,j}$.

  \ENDFOR \label{row:loop:end}
     \FOR{$i=1,\ldots,m$} \label{line:end:of:third}
         \STATE $\epsilon_{i,n} \leftarrow r_i \mod 2$ \label{line:last:column}
         \STATE $r_i \leftarrow r_i - \epsilon_{i,n}$
         \STATE $r_i \leftarrow r_i / 2$
         \STATE $c_n \leftarrow c_n - \epsilon_{i,n}$
      \ENDFOR\label{line:last:three}
      
       \STATE $c_n \leftarrow  c_n / 2 $. 

	\STATE $\epsilon \leftarrow \sigma_R^{-1}\circ {\bf \epsilon}$  (Apply permutation to rows)
	\STATE $\epsilon \leftarrow \sigma_C^{-1}\circ {\bf \epsilon}$  (Apply permutation to columns)

	\STATE $t \leftarrow t+2^b \epsilon$. \label{line:combine}
  \ENDFOR
\end{algorithmic}
\end{algorithm}

\begin{proposition}
Algorithm~\ref{approximate algorithm} requires on average $O(m^2 n\, \log M)$ random bits, with $O(m\, n^2\, M\log^2 M)$ arithmetic operations. 
\end{proposition}
\begin{proof}
By Lemma~\ref{runtime:lemma}, each entry is rejected at most an expected $O(m)$ number of times. 
The rejection function needs to be computed at worst $O(m\,n\,\log M)$ times, with a cost of $O(n\, M\, \log M)$ for performing an $n$-fold convolution via fast Fourier transforms. 
\end{proof}

\subsection{$2\times n$ tables}\label{main:sect:two}

For a more presently practical \emph{exact} sampling algorithm, we consider the application of PDC to the case of $2 \times n$ tables. 
Like rejection sampling, PDC is adaptable, with almost limitless possibilities for its application, and in this particular case there is a simple division which yields a simple and practical algorithm. 

When there are only two rows, the distribution of $\xi_{1,j}$ given $\xi_{1,j}+\xi_{2,j}=c_j$ is uniform on $\{0,\ldots,c_j\}$, so we avoid the geometric distribution altogether. 
This yields the following simple algorithm, which does not require the computation of any rejection functions nor a lookup table.

\begin{algorithm}[H]
\caption{generating a uniformly random $2 \times n$ $(r,c)$-contingency table.}
\label{PDC DSH 2 by n}
\begin{algorithmic}[1]
\FOR {$j=1,\ldots,n-1$}
\STATE choose $x_{1,j}$ uniformly from $\{0,\ldots,c_j\}$
\STATE let $x_{2,j} = c_j - x_{1,j}$
\ENDFOR
\STATE let $x_{1,n} = r_1 - \sum_{j=1}^{n-1} x_{1,j}$\\[.03cm] 
\STATE let $x_{2,n} = r_2 - \sum_{j=1}^{n-1} x_{2,j}$
\IF {$x_{1,n}<0$ or $x_{2,n}<0$}
  \STATE restart from Line 1
\ENDIF
\RETURN $x$
\end{algorithmic}
\end{algorithm}

To summarize Algorithm~\ref{PDC DSH 2 by n}: sample entries in the top row one at a time, \emph{except $\xi_{1,n}$}, uniformly between 0 and the corresponding column sum $c_j$. The rest of the table is then determined by these entries and the prescribed sums; as long as all entries produced in this way are non-negative, we accept the result to produce a uniformly random $(r,c)$-contingency table.

\begin{theorem}\label{two:row:theorem}
Let $U_1, U_2, \ldots, U_{n-1}$ denote independent uniform random variables, with $U_j$ uniform over the set of integers $\{0, 1, \ldots, c_j\}$, $j=1,\ldots, n-1$, and define
\[ p_n := \P\big(U_1 + \ldots + U_{n-1} \in [r_1-c_n, r_1]\big). \]
Algorithm~\ref{PDC DSH 2 by n} produces a uniformly random $2\times n$ $(r,c)$-contingency table, with expected number of rejections $O(1/p_n)$. 
\end{theorem}

\begin{corollary}\label{equal density 2 by n}
When the row sums are equal and the column sums are equal, the expected number of rejections before Algorithm~\ref{PDC DSH 2 by n} terminates is $O(n^{1/2})$.  
\end{corollary}

There is a key observation at this point, which is that exponential random variables share the same property that $\xi_{1,j}$ given $\xi_{1,j}+\xi_{2,j}=c_j$ is uniform over the \emph{interval} $[0,c_j]$. 
The algorithm for real--valued $2 \times n$ tables is presented below in Algorithm~\ref{PDC DSH 2 by n real}, and is essentially the same as Algorithm~\ref{PDC DSH 2 by n}.

\begin{algorithm}[H]
\caption{generating a uniformly random $2 \times n$ real--valued $(r,c)$-contingency table.}
\label{PDC DSH 2 by n real}
\begin{algorithmic}[1]
\FOR {$j=1,\ldots,n-1$}
\STATE choose $x_{1,j}$ uniformly from $[0,c_j]$
\STATE let $x_{2,j} = c_j - x_{1,j}$
\ENDFOR
\STATE let $x_{1,n} = r_1 - \sum_{j=1}^{n-1} x_{1,j}$\\[.03cm] 
\STATE let $x_{2,n} = r_2 - \sum_{j=1}^{n-1} x_{2,j}$
\IF {$x_{1,n}<0$ or $x_{2,n}<0$}
  \STATE restart from Line 1
\ENDIF
\RETURN $x$
\end{algorithmic}
\end{algorithm}

\subsection{Approximate sampling of real--valued tables}
\label{sect:real}

By Lemma \ref{lemma:uniform:real}, a uniformly random $(r,c)$-contingency table with real-valued entries has distribution
\[\L\big(\xi\mid E_{r,c}\big),\]
where $\xi=(\xi_{ij})$, $1 \leq i \leq m, 1 \leq j \leq n,$ is an $m\times n$ matrix of independent exponential random variables with parameter $\lambda_{ij} = -\log(1-p_{ij})$.
An exponential random variable $E(\lambda_{ij})$ can be decomposed into a sum
\[E(\lambda_{ij})=A(\lambda_{ij})+G(p_{ij}),\]
where $A(\lambda)$ is a random variable with density \[f(x) = \lambda e^{-\lambda x} (1-e^{-\lambda})^{-1}, \qquad 0 \leq x \leq 1,\] and $G(p)$ is a geometric random variable with probability of success at each trial given by $p$,  independent of $A(\lambda)$. 

One could attempt to adapt the bit-by-bit rejection of Algorithm~\ref{mainalgorithm} in the continuous setting to fractional parts, however, instead of counting the number of residual tables, one would have to look at the \emph{density} of such tables with respect to Lebesgue measure in $\R^{m\times n}$. 
Instead, we champion an approximate sampling algorithm analogous to Algorithm~\ref{approximate algorithm} which circumvents such calculations.

For each $1\leq i \leq m, 1 \leq j \leq n$, let $Y_{i,j}(\lambda_{i,j})$ denote independent exponential random variables with parameter $\lambda_{i,j}$, and let $Y_{i,j}'(\lambda_{i,j})$ denote a random variable with distribution 
\[\L(Y_{i,j}'(\lambda_{i,j})) = \L\left(Y_{i,j}(\lambda_{i,j}) \middle| \sum_{\ell=1}^m Y_{\ell,j}(\lambda_{i,j}) = c_j\right). \]
In what follows we will take $\lambda_{i,j} = -\log\left(\frac{c_j}{m+c_j}\right)$, i.e., $\lambda_{i,j}$ does not vary with parameter~$i$, and let $q_{i,j} = \frac{c_j}{m+c_j}$.
Let $\lfloor Y_{i,j,s}''\rfloor$ denote a random variable with distribution
\[\L(\lfloor Y_{i,j,s}''\rfloor) = \L\left(\lfloor Y_{i,j}\rfloor \middle| \sum_{i=1}^s \xi_{i,j}(q_{i,j}) + \sum_{i=s+1}^m Y_{i,j}(\lambda_{i,j}) = c_j\right). \]

A suggested rejection function is then
\begin{align*}
& G(i,j,x, r, c) \\
 & := \Pr\left( \sum_{\ell=1}^{j-1} \xi'_{i,\ell}(q_\ell) + 2\lfloor Y''_{i,j,i}(\lambda_{i,j})\rfloor + \sum_{\ell=j+1}^{n} Y'_{i,\ell}(\lambda_\ell) \in d(r_i-x) \right) \\
 & \ \ \ \times \Pr\left(\sum_{\ell=1}^i \lfloor Y_{\ell,j} \rfloor + \sum_{\ell=i+1}^m Y_{\ell,j} \in d(c_j - x) \right), \qquad  \mbox{for } x \in [0,1].\\
 \end{align*}

In Algorithm~\ref{alternative fractional real} below, we let FractionExp$(\lambda)$ denote the distribution of the fractional part of an exponential random variable with parameter $\lambda$.

\begin{algorithm}[H]
\caption{Generation of fractional parts of uniformly random $(r,c)$-real-valued table}
\label{alternative fractional real}

\begin{algorithmic}[1]

   \STATE Let $\sigma_R$ denote any permutation such that $\sigma_R \circ r$ is in increasing order.
   \STATE Let $\sigma_C$ denote any permutation such that $\sigma_C \circ c$ is in increasing order.

   \STATE $r \leftarrow \sigma_R \circ r$.   
   \STATE $c \leftarrow \sigma_C \circ c$.

  \FOR {$j=1,\ldots,n-1$}\label{column:loop}

      \FOR {$i=1,\ldots,m-1$}\label{row:loop}
      \STATE $\lambda_j \leftarrow -\log\left(\frac{c_j}{m+c_j}\right)$.

          \STATE $\epsilon_{i,j} \leftarrow$ FractionExp$(\lambda_j)$\label{line:start-of-loop:fractional}

  \IF{$U > \frac{G(i,j,\epsilon_{i,j},r,c)}{\max_{x\in[0,1]} G(i,j,x,r,c)}$} \label{reject:fractional}
     \STATE \textbf{goto} Line~\ref{line:start-of-loop:fractional}.
     \ENDIF 
  \STATE $r_i \leftarrow r_i - \epsilon_{i,j}$.      
  \STATE $c_j \leftarrow c_j - \epsilon_{i,j}$.      
  \ENDFOR \label{column:loop:end}
  
  \STATE $\epsilon_{m,j} \leftarrow  \{c_j\}$ \label{line:last:column}
  \STATE $c_j \leftarrow c_j - \epsilon_{m,j}$.
  
  \STATE $r_m \leftarrow r_m - \epsilon_{m,j}$.

  \ENDFOR \label{row:loop:end}
     \FOR{$i=1,\ldots,m$} \label{line:end:of:third}
         \STATE $\epsilon_{i,n} \leftarrow \{r_i\} $ \label{line:last:column}
         \STATE $r_i \leftarrow r_i - \epsilon_{i,n}$
         
         \STATE $c_n \leftarrow c_n - \epsilon_{i,n}$
      \ENDFOR\label{line:last:three}

	\STATE $\epsilon \leftarrow \sigma_R^{-1}\circ {\bf \epsilon}$  (Apply permutation to rows)
	\STATE $\epsilon \leftarrow \sigma_C^{-1}\circ {\bf \epsilon}$  (Apply permutation to columns)

	\STATE \textbf{return} $\epsilon$. 

\end{algorithmic}
\end{algorithm}

\begin{algorithm}[H]
\caption{Generation of approximately uniform random real-valued $(r,c)$-table}
\label{mainalgorithm real}
\begin{algorithmic}[1]
  \STATE $\epsilon \leftarrow$ output of Algorithm~\ref{alternative fractional real} with $(r,c)$ input.
  \STATE $r'_i \leftarrow r_i - \sum_{j=1}^n \epsilon_{i,j}$, $i=1,\ldots,m$.
  \STATE $c'_j \leftarrow c_j - \sum_{i=1}^m \epsilon_{i,j}$, $j=1,\ldots,n$.
  \STATE $t \leftarrow$ output of Algorithm~\ref{approximate algorithm} with $(r',c')$ input.
\STATE Return  $t+\epsilon$.\label{real:line:combine}
\end{algorithmic}
\end{algorithm}

%% file: ctablesMethod.tex
% !TEX root = ctables.tex

\section{Background}\label{sect:method}
\subsection{Contingency Tables}

\begin{definition}
Let $r=(r_1,\ldots,r_m)$ and $c=(c_1,\ldots,c_n)$ be vectors of non-negative integers, with $r_1+\cdots+r_m=c_1+\cdots+c_n$. 
An \emph{$(r,c)$-contingency table} is an $m\times n$ matrix $\xi=(\xi_{ij})_{1\leq i \leq m,1 \leq j \leq n}$ with non-negative integer entries, whose row sums $r_i=\sum_j\xi_{ij}$ and column sums $c_j=\sum_i\xi_{ij}$ are prescribed by $r$ and $c$. 
Let $N=\sum_ir_i=\sum_jc_j$ be the sum of all entries. %, and let $s=S/mn$ be the average entry size, which we call the \emph{density} of the table.
%Let $\gamma=\max_j|c_j-S/n|$ denote the variation in the columns.
%Let $\gamma=\max_i|r_i-S/m|$ denote the variation in the rows.
% we don't use this max variation any more - jz
We denote the set of all $(r,c)$-contingency tables by the set%Given a collection of row sums $r = (r_1, \ldots, r_m)$ and column sums $c = (c_1, \ldots, c_n)$, we define the set
\[ E \equiv E_{r,c} = \left\{ \{\xi_{ij}\}_{1\leq i \leq m, 1 \leq j \leq n} \in \N_0^{m\times n}: \ \begin{array}{l} \sum_{\ell=1}^n \xi_{i,\ell} = r_i ~\forall ~ 1\leq i\leq m, \\ \\  \sum_{\ell=1}^m \xi_{\ell,j} = c_j ~\forall ~1 \leq j \leq n \end{array}\right\}. \]
% 16/1/2015: changed both summation variables to \ell to avoid conflict with k in the next line, changed one of the commas to \forall - JZ
\end{definition}

%We say that $X$ is geometrically distributed with parameter $p$ when it has point probabilities of the form %By geometric distribution we mean the distribution with domain $k=0,1,2,\ldots$, parameter $p$, such that if $X$ is a geometric random variable with parameter $p$, then 
%$\P(X = k) = (1-p)^k\, p$, for $k=0,1,\ldots$.  

The set $E$ is also known as the transportation polytope, see for example \cite{barvinok2009asymptotic}, and the measure of interest is typically the uniform measure over the set of \emph{integral} points.   % or over all real--valued points.

We shall also consider real--valued tables with real--valued row sums and column sums, in which case the set $E$ is described by 
\[ E_{r,c} = \left\{ \{\xi_{ij}\}_{1\leq i \leq m, 1 \leq j \leq n} \in \R_{\geq 0}^{m\times n}: \ \begin{array}{l} \sum_{\ell=1}^n \xi_{i,\ell} = r_i ~\forall ~ 1\leq i\leq m, \\ \\  \sum_{\ell=1}^m \xi_{\ell,j} = c_j ~\forall ~1 \leq j \leq n \end{array}\right\}. \]
Since this set has infinite cardinality, we instead consider the density of the set with respect to Lebesgue measure on $\R^{m\times n}.$ %,$ denoted by $d\mu(r,c)$. 

\subsection{Rejection Sampling}

A random contingency table can be described as the joint distribution of a collection of \emph{independent} random variables which satisfy a condition.
Many combinatorial structures follow a similar paradigm, see for example \cite{duchon2004boltzmann, IPARCS} and the references therein.  
Let $\X = (X_{i,j})_{1 \leq i \leq m, 1 \leq j \leq n}$ denote a collection of independent random variables, forming the entries in the table.
%Given a collection of row sums $r = (r_1, \ldots, r_m)$ and column sums $c = (c_1, \ldots, c_n)$, we define the event
%\[ E \equiv E_{r,c} = \left\{ \sum_{\ell=1}^m X_{\ell,j} = c_j, \sum_{k=1}^n X_{i,k} = r_i, 1\leq i\leq m, 1 \leq j \leq n\right\}. \]
Given a collection of row sums $r = (r_1, \ldots, r_m)$ and column sums $c = (c_1, \ldots, c_n)$, the random variable $\X',$ with distribution
\begin{equation}\label{main}
\L(\X') := \L(\, (X_{1,1}, \ldots, X_{m,n})\, |\, E ),
\end{equation}
 is then representative of some measure over the set $E$.
A general framework for the random sampling of joint distributions %stochastic processes 
of this form is contained in \cite{PDCDSH}. %, from which our algorithms are derived.

When the set $E$ has positive measure, the simplest approach to random sampling of points from the distribution~\eqref{main} is to sample from $\L(\X)$ repeatedly until $\X \in E$; this is a special case of \emph{rejection sampling} \cite{Rejection}, see also \cite{devroye}, which we refer to as \emph{hard rejection sampling}.
The number of times we must repeat the sampling of $\L(\X)$ is geometrically distributed with expected value $\P(\X\in E)^{-1}$, which may be prohibitively large.  % see for example \cite{devroye}.

Beyond hard rejection sampling, one must typically exploit some special structure in $\L(\X)$ or $E$ in order to improve upon the number of repetitions.
Indeed, for contingency tables, we can easily improve upon the na\"{\i}ve hard rejection sampling of the entire table by applying hard rejection sampling to each column independently, with total expected number of repetitions $\sum_j (\P( \sum_{i=1}^m X_{i,j} = c_j))^{-1}$; then, after each column has been accepted, we accept the entire set of columns if every row condition is satisfied. %, see Algorithm~\ref{alg:lucky}.
%Even this more modularized approach has a prohibitively small probability of success for values of $m$ and $n$ larger than 10.
This is the approach championed in Section~\ref{sect:other} for more general distributions on the entries. 

Finally, we note that hard rejection sampling fails when $\P(\X\in E) = 0$; in this case, the target set of interest typically lies on a lower--dimensional subspace of the sample space, and it is not apparent how one could generally adapt hard rejection sampling.
In terms of contingency tables, we may wish to sample random real--valued points which satisfy the conditions; if the sums of random variables have densities, the conditioning is well--defined, even though the probability of generating a point inside the table is 0; see \cite{PDCDSH}.

\ignore{
Algorithm~\ref{mainalgorithm fractional real} is an application of what we call \emph{soft} rejection sampling combined with recursive PDC at the end of each iteration, as described in \cite{PDC}.  
Essentially, instead of sampling from the entire set of random variables until a condition is satisfied, one partitions the sample space into two pieces, samples one piece in proportion to its prevalence in the target set, and, conditional on this first piece, the second piece is a smaller version of the original problem. %all but one, and then accepts the unique completion of the entire set \emph{in proportion} to its likelihood in the target distribution, rather than randomly simulating it.
%This approach is equally well suited to collections of discrete and continuous random variables.
This approach was used in~\cite{PDC} to obtain an asymptotically fastest algorithm for uniform sampling of integer partitions, by sampling the smallest bits, next smallest bits, etc., of the multiplicities of the part sizes. 
}

\subsection{Uniform Sampling}

We now summarize some properties that we utilize to prove our main results.

\begin{lemma}
\label{lemma:uniform}
Suppose $\X=(X_{ij})_{1 \leq i \leq m, 1 \leq j \leq n}$ are independent geometric random variables with parameters $p_{ij}$.  
If $p_{ij}$ has the form $p_{ij} = 1 - \alpha_i \beta_j$, then $\X$ is uniform restricted to $(r,c)$-contingency tables.
\end{lemma}
\begin{proof}
For any $(r,c)$-contingency table $\xi$,
\[\P\big(\X=\xi\big)
=\prod_{i,j}\P\big(X_{ij}=\xi_{ij}\big)
=\prod_{i,j}(\alpha_i \beta_j)^{\xi_{ij}}(1-\alpha_i \beta_j)
=\prod_i\alpha_i^{r_i} \prod_j\beta_j^{c_j}\prod_{i,j}\big(1-\alpha_i \beta_j\big).
\]
Since this probability does not depend on $\xi$, it follows that the restriction of $X$ to $(r,c)$-contingency tables is uniform.
\end{proof}

For $j=1,\ldots,n$, let $C_j=(C_{1j},\ldots,C_{mj})$ be independent random vectors with distribution given by $(X_{1j},\ldots,X_{mj})$ conditional on $\sum_iX_{ij}=c_j$; that is,
\[\P\big(C_j=(\xi_{1j},\ldots,\xi_{mj})\big)=\frac{\P\big(X_{1j}=\xi_{1j},\ldots,X_{mj}=\xi_{mj}\big)}{\P\big(\sum_iX_{ij}=c_j\big)}\]
for all non-negative integer vectors $\xi_j$ with $\sum_i\xi_{ij}=c_j,$ and 0 otherwise.

\begin{lemma}
The conditional distribution of $C=(C_1,\ldots,C_n)$ given $\sum_jC_{ij}=r_i$ for all $i$ is that of a uniformly random $(r,c)$-contingency table.
\label{lem:luckyuniformity}
\end{lemma}
\begin{proof}
For any $(r,c)$-contingency table $\xi$, $\P(C=\xi)$ is a constant multiple of $\P(x=\xi)$.
\end{proof}

%Now that we have established the uniformity of the model, we formulate an algorithm which generates random samples uniformly from the set $E$, and 
Our next lemma shows how to select the parameters $p_{ij}$ in order to optimize the probability of generating a point inside of $E$. 
It also demonstrates a key difficulty in the sampling of points inside $E$, which is that we can only tune the probabilities $p_{i,j}$ to optimally target rows sums or column sums, but not both simultaneously. 

\begin{lemma}\label{lemma:q}
 Suppose $\X$ is a table of independent geometric random variables, where $X_{i,j}$ has parameter $p_{ij}=m/(m+c_j)$, $1\leq i \leq m$, $1 \leq j \leq n$.  
% Recall that $\gamma=\max_i|r_i-S/m|$.  
Then the expected columns sums of $\X$ are $c$, and the expected row sums of $\X$ are $N/m$. % and thus uniformly within $\gamma$ of $c$.
\end{lemma}
\begin{proof}
For any $j=1,\ldots,n$,
\[\sum_{i=1}^m\E\big[x_{ij}\big]=\sum_{i=1}^m\bigg(\frac{m+c_j}{m}-1\bigg)=c_j.\]
Similarly, for any $i=1,\ldots,m$,
\[\sum_{j=1}^n\E\big[x_{ij}\big]=\sum_{j=1}^n \bigg(\frac{m+c_j}{m}-1\bigg)=\frac Nm.\qedhere\]
\end{proof}

%By the central limit theorem, $\sum_{j=1}^n X_{ij}$ is approximately normally distributed, with expected value $S/m$ and standard deviation $\sqrt{S/m}$.
%Thus, as long as $\gamma=\max_i|r_i-S/m|$ is $O(\sqrt{S/m})$, we can expect the algorithm to perform well.  
\begin{remark}\label{conditional:independence}
Note that entries in different rows (columns, resp.) are conditionally independent. This means that we can separately sample all of the rows (columns, resp.) independently until all of the row (column, resp.) conditions are satisfied. 
%Then we reject if any of the column (row, resp.) conditions are violated, as is outlined in Algorithm~\ref{alg:lucky}.  
It also means that once all but one column (row, resp.) are sampled according to the appropriate conditional distribution, the remaining column (row, resp.) is uniquely determined by the constraints. 
\end{remark}

For real--valued tables, the calculations are analogous and straightforward.

\begin{lemma}
\label{lemma:uniform:real}
Suppose $\X=(X_{ij})_{1 \leq i \leq m, 1 \leq j \leq n}$ are independent exponential random variables with parameters $\lambda_{i,j} := -\log(1-p_{ij})$.  
If $p_{ij}$ has the form $p_{ij} = 1 - \alpha_i \beta_j$, then $\X$ is uniform restricted to real--valued $(r,c)$-contingency tables.
\end{lemma}
\begin{proof}
For any real--valued $(r,c)$-contingency table $\xi$,
\begin{align*}
\P\big(\X\in d\xi\big) & = \prod_{i,j}\P\big(X_{ij}\in d\xi_{ij}\big)=\prod_{i,j}\lambda_{i,j}\ e^{-\lambda_{i,j} \xi_{i,j}}  \\
 & = \prod_{i,j} \big(\lambda_{i,j}\big) \prod_{i,j} \left(\alpha_i \beta_j\right)^{\xi_{ij}}
=\prod_i\alpha_i^{r_i} \prod_j\beta_j^{c_j}\prod_{i,j}\big(\lambda_{i,j}\big).
\end{align*}
Since this probability does not depend on $\xi$, it follows that the restriction of $X$ to real--valued $(r,c)$-contingency tables is uniform.
\end{proof}

\begin{lemma}\label{lemma:q}
 Suppose $\X$ is a table of independent exponential random variables, where $X_{i,j}$ has parameter $\lambda_{i,j} = -\log\left(1-p_{ij}\right)$, $1\leq i \leq m$, $1 \leq j \leq n$.  
% Recall that $\gamma=\max_i|r_i-S/m|$.  
Then the expected columns sums of $\X$ are $c$, and the expected row sums of $\X$ are $N/m$. % and thus uniformly within $\gamma$ of $c$.
\end{lemma}

\begin{remark}
{\rm 
Remark~\ref{conditional:independence} applies when the $X_{i,j}$'s are independent and have any discrete distribution, not just geometric.
If the $X_{i,j}$'s are independent \emph{continuous} random variables such that each of the sums $\sum_{i=1}^m X_{i,j}$ for $j=1,2,\ldots,n$ and $\sum_{j=1}^n X_{i,j}$ for $i=1,2,\ldots,m$ has a density, then the same conditional independence of entries in different rows are conditionally independent, and we can separately sample all of the rows independently until all of the row conditions are satisfied. 
}\end{remark}

\ignore{ % TODO: Work out details.
\begin{lemma}
\label{lemma:uniform}
Suppose $\X=(X_{ij})_{1 \leq i \leq m, 1 \leq j \leq n}$ are independent exponential random variables with parameters $-\log(p_{ij})$.  
If $p_{ij}$ has the form $p_{ij} = 1 - \alpha_i \beta_j$, then $\X$ is uniform restricted to $(r,c)$-contingency tables.
\end{lemma}
}

\subsection{Probabilistic Divide-and-Conquer}\label{sect:PDC}
\ignore{
Our main tool is probabilistic divide-and-conquer
The essence of \emph{probabilistic divide-and-conquer}  --- PDC  --- is random sampling of conditional distributions, which, \emph{when appropriately pieced together}, represent a sample from the target distribution.    
We assume throughout that the target object $S$ may be expressed as a pair $(A,B)$, where
\begin{equation}\label{def AB}
  A \in \mathcal{A}, \ B \in \mathcal{B} \ \mbox{ have given distributions},
\end{equation}
\begin{equation}\label{indep AB}
  A , B \ \mbox{ are independent},
\end{equation}
\begin{equation}\label{def h}
h: \mathcal{A} \times \mathcal{B} \to \{0,1 \} 
\end{equation}
$$
 \mbox{ satisfies   }  p := \e h(A,B) \in (0,1],\footnote{
 The requirement that $p>0$ is  a choice we make here, for the sake of simpler exposition;  there are $p=0$ examples where divide-and-conquer is useful, see \cite{PDCDSH}.
In cases where $p=0$,
 the conditional distribution, apparently specified by \eqref{def S}, needs further specification --- this is known as
 Borel's paradox.
 }
 $$
where, of course, we also assume that $h$ is measurable, and
\begin{equation}\label{def S}
S \in \mathcal{A} \times \mathcal{B}  \mbox{ has distribution  }  \L(S) = \L(\,  (A,B)\, | \,h(A,B) =1),
\end{equation}
i.e., the law of $S$ is the law of the independent pair $(A,B)$ \emph{conditional on} 
having $h(A,B)=1$.
}

The main utility of PDC is that when the conditioning event has positive probability, it can be applied to any collection of random variables. 
When the conditioning event has probability 0, some care must be taken, although in our current setting we can apply~\cite[Lemma~2.1, Lemma~2.2]{PDCDSH} directly.  
%but our examples are applications as long as the conditioning event is of the form $\{T = k\}$ for some random variable $T$ with bounded density, and $k \in \text{range}(T)$, then PDC applies. 

We now recall the main PDC lemma, which is simplest to state in terms of a target event $E$ of positive probability. 
Suppose that $\cX$ and $\cY$ are sets, and let $\cZ$ be a subset of $\cX\times\cY$. Let $A$ and $B$ be probability measures on $\cX$ and $\cY,$ respectively. The goal of PDC is to provide a technique to sample from the distribution of the cartesian product of $A$ and $B$ restricted to $\cZ$.

\begin{theorem}[Probabilistic Divide-and-Conquer \cite{PDC}]
\label{thm:pdc}
For each $a\in\cX$, let $\cY_a=\{b\in\cY:(a,b)\in\cZ\}$, where $\P((A,B) \in E)>0$. 
%Pick any constant $C\ge\sup_x B(\cY_x)$, and let
Let 
\[\L(A') := \L\left(A\,\big|\,  (A,B) \in E\right), \]
\[\L(B'_a) := \L\big(B\,\big|\,\cY_a\big).\]
%Then, $\big(A', B'_{A'}\big)$ has the same distribution as $\big((A,B)\,\big|\,\cZ\big)$.
Then, $\L\big(A', B'_{A'}\big) = \L\big((A,B)\,\big|\,\cZ\big)$.
\end{theorem}

A similar theorem holds when $\P((A,B) \in E) = 0$, under some simple conditions. 
\begin{theorem}[Probabilistic Divide-and-Conquer for certain events of probability 0 \cite{PDCDSH}]
\label{thm:pdcdsh}
For each $a\in\cX$, let $\cY_a=\{b\in\cY:(a,b)\in\cZ\}$, where $\P((A,B) \in E)=0$. 
Suppose there is a random variable $T$ with a density, and $k \in \mbox{range}(T)$ such that $ \P((A,B) \in E) = \P(T = k)$.  
Suppose further that for each $a \in \cX$, there is a random variable $T_a$, either discrete or with a density, and $k \in \mbox{range}(T_a)$, such that $\P(\cY_a = b) = \P(T_a = k)$. 
%Pick any constant $C\ge\sup_x B(\cY_x)$, and let
Let 
\[\L(A') := \L\left(A\ |\  (A,B) \in E\right), \]
\[\L(B'_a) := \L\big(B\,\big|\,\cY_a\big).\]
%Then, $\big(A', B'_{A'}\big)$ has the same distribution as $\big((A,B)\,\big|\,\cZ\big)$.
Then, $\L\big(A', B'_{A'}\big)=\L\big((A,B)\,\big|\,\cZ\big)$.
\end{theorem}

\ignore{
We begin by defining some notation that will allow us to easily manipulate such products and restrictions. %Allowing $X$ and $Y$ to denote random variables with the given distributions, l
Let $(A,B)$ denote a random variable on $\cX\times\cY$ with
\[\P\big((A,B)=(a,b)\big)=\P\big(A=a\big)\P(B=b\big).\]
For a random variable $Z$ on a measurable space $(S,\cF(S))$, and any measurable subset $T\in\cF(S)$, let $(Z\mid T)$ denote the random variable on $(T,\cF(T))$ with
\[\P\big((Z\mid T)=z\big)=\frac{\P(Z=z)}{\P(Z\in T)}.\]
Under this notation, the object we wish to study can be written $((A,B)\mid\cZ)$.
}
%Finally, f
Our recommended approach to sample from $\L(A')$ is rejection sampling.  For any measurable function $p:S\rightarrow[0,1]$, let $\L(Z\mid U<p(Z))$ denote the first coordinate of $\L((Z,U)\mid U<p(Z))$, where $U$ is a uniform random variable on $[0,1]$ independent of $Z$. This allows concise descriptions of distributions resulting from rejection sampling.

\begin{lemma}[\cite{PDCDSH}]\label{pdc:rejection}
Under the assumptions of Theorem~\ref{thm:pdc} or Theorem~\ref{thm:pdcdsh}, pick any finite constant $C\ge\sup_x B(\cY_x)$.  We have
\[\L(A')= \L\left(A\ \bigg|\  U<\frac{B(\cY_a)}{C}\right).\]
That is, to sample from $\L(A')$, we first sample from $\L(A)$, and then reject with probability $1 - B(\cY_a) / C$. 
\end{lemma}

In our applications of PDC that follow, we indicate the use of PDC by specifying distributions $\L(A)$ and $\L(B)$ and an event $E$ such that the measure of interest is $\L\big( (A,B) \,\big|\, (A,B)\in E\big)$.  %For example, to sample from $(

% ------------------- Everything below this line is ignored -----------------------
\ignore{
\begin{algorithm}[H]
\caption{Na\"{\i}ve (suboptimal) rejection sampling of a uniformly random $(r,c)$-contingency table.}
\label{alg:lucky}
\begin{algorithmic}[1]
\FOR {$i=1,\ldots,n$}
  \FOR {$j=1,\ldots,m$}
    \STATE generate $x_{ij}$ from Geometric($\frac{m}{m+c_j}$)
  \ENDFOR
  \IF {$\sum_ix_{ij}=c_j$}
    \STATE let $C_j=(x_{1j},\ldots,x_{mj})$
  \ELSE
    \STATE \textbf{restart} from Line 2
  \ENDIF
\ENDFOR
\IF {$\sum_jx_{ij}=r_i$ for all $i$}
  \RETURN $(C_1,\ldots,C_n)$
\ELSE
  \RESTART from Line 1
\ENDIF
\end{algorithmic}
\end{algorithm}
}
\ignore{
\begin{lemma}
 Suppose $\X$ is a matrix of independent geometric random variables, where $X_{i,j}$ has parameter $p_{ij}=m/(m+c_j)$, $1\leq i \leq m$, $1 \leq j \leq n$.  
 Recall that $\gamma=\max_i|r_i-S/m|$.  
Then the expected column sums of $\X$ are $c$, and the expected rows sums of $\X$ are $N/m$ and thus uniformly within $\gamma$ of $r$.
\end{lemma}
\begin{proof}
For any $j=1,\ldots,n$,
\[\sum_i\E\big[x_{ij}\big]=\sum_i\bigg(\frac{m+c_j}{m}-1\bigg)=c_j.\]
Similarly, for any $i=1,\ldots,m$,
\[\sum_j\E\big[x_{ij}\big]=\sum_j\bigg(\frac{m+c_j}{m}-1\bigg)=\frac Nm.\qedhere\]
\end{proof}
} %end ignore

% end ignore

%end ignore

% end ignore

% end ignore

% end ignore

% end ignore

% end ignore

%% file: ctablesRuntime.tex
% !TEX root = ctables.tex

\section{Proof of uniformity}
\label{sect:proofs}

\begin{lemma}\label{uniform:lemma}
Algorithm~\ref{mainalgorithm} produces a uniformly random $(r,c)$-contingency table.
\end{lemma}

\begin{proof}
We need to show that the function $f$ in Section~\ref{int:tables} %equation~\eqref{f} 
samples $\epsilon_{i,j}$ in its right proportion, namely,  %converts the distribution $\L(\epsilon_{i,j})$ into
\[\L(\epsilon_{i,j} | E, (\epsilon_{\ell,1})_{\ell=1,\ldots,m},  \ldots, (\epsilon_{\ell,j-1})_{\ell=1,\ldots,m}, (\epsilon_{\ell,j})_{\ell=1,\ldots,j-1}),\]
i.e., that we are getting the correct distribution by filling in the bits one at a time.  %This is the essential idea in~\cite{Patefield} for sampling of the whole entry of the table one at a time, and here we have adapted it for the random sampling of bits. 

Let us start by considering the random sampling of the least significant bit, say $\epsilon_{1,1}$, of the top left entry of the table. 
Suppose we sample from %We first sample from 
\[ \L(\epsilon_{1,1}) = \L\left(\Bern\left(\frac{q_1}{1+q_1}\right)\right),\]
 and then reject according to the correct proportion in the conditional distribution. %according to Line~\ref{reject}. 
%As per rejection sampling, w
We accept this sample in proportion to $\P(E | \epsilon_{1,1})$, which is given by 
\[ \P(E | \epsilon_{1,1}=k) = \Sigma\left( \begin{array}{l}(r_{1}-k,\ldots), \\ (c_{1}-k,\ldots), \\ \mathcal{O}_{1,1}\end{array} \right) \cdot\, q_1^{c_1-k} (1-q_1^2)\, (1-q_1)^{m-1}\prod_{j=2}^n q_j^{c_j} \big(1- q_j\big)^m. \]
Normalizing by all terms which do not depend on $k$ gives % Equation~\eqref{rejection:ij}. 
\[ \P(E | \epsilon_{1,1}=k) \propto \Sigma\left( \begin{array}{l}(r_{1}-k,\ldots), \\ (c_{1}-k,\ldots), \\ \mathcal{O}_{1,1}\end{array} \right) \cdot\, q_1^{-k}, \]
and since $\P(\epsilon_{1,1} = k) \propto q_1^k$, Equation~\eqref{rejection:ij} follows for the case $i=j=1$.  The same reasoning applies for $1 \leq i \leq m-2$ and $1 \leq j \leq n-2$.

Next, consider the case when we sample from $\L(\epsilon_{m-1,j}) = \L\left(\Bern\left(\frac{q_j}{1+q_j}\right)\right)$, $j=1,2,\ldots,n-1$. 
We accept this sample in proportion to $\P(E | \epsilon_{1,1}, \ldots, \epsilon_{m-1,j})$.  Note, however, that given the least significant bits of the first $m-1$ entires of any given column, the bit $\epsilon_{m,j}$ is determined by the column sum. 
Let $k := \epsilon_{m-1,j}$ and $k' := \epsilon_{m,j}$ be such that $c_j - k - k'$ is even. Applying PDC deterministic second half, then implies the rejection function is proportional to 
\[ \P(E | \epsilon_{1,j},\ldots,\epsilon_{m-1,j}) \propto \Sigma\left( \begin{array}{l} (\ldots,r_{m-1}-k,r_m-k'), \\ (\ldots,c_{j}-k-k',\ldots), \\ \mathcal{O}_{m,j}\end{array}\right)\,\cdot\, q_j^{c_j-k-k'} (1-q_j^2)^{m-2}\, (1-q_j)^{2}. \]
Since $\P(\epsilon_{m-1,j} = k) \propto q_j^k$, Equation~\eqref{rejection:mj} follows.

The other cases are similar, and simply require keeping track of which other bits are determined under the conditioning. 
We have thus demonstrated the ability to exactly sample from the least significant bit of each entry of the table.

We now proceed by induction. 
Let $\epsilon_0$ denote the least significant bits of all entries in the table, with $\epsilon_i$, $i=1,2,\ldots$ denoting the $(i+1)$th least significant bits of the table. 
The first iteration of the algorithm generates an element from the distribution $\L(\epsilon_0 \,|\, E_{r,c})$, and updates the values of $r$ and $c$ to $r'$ and $c'$, respectively, using the values in $\epsilon_0$. 
The second iteration of the algorithm generates an element from the distribution $\L(\epsilon_0' \,|\, E_{r',c'})$, which is distributed as $\L(\epsilon_1 \,|\, E_{r,c}, \epsilon_0)$. 
By Theorem~\ref{thm:pdc}, the expression for $t$ in Line~\ref{line:combine} is distributed as $\L(\epsilon_0 + 2\epsilon_1 \,|\, E_{r,c}).$
%and the observation that, for a geometrically distributed random variable $Z$ with parameter $p$, the parity bit, say $\epsilon$, is independent of $(Z-\epsilon)/2$, which is geometrically distributed with parameter $(1-p)^2$. 

Assume that after the first $b-1$ iterations of the algorithm, the expression for $t$ in Line~\ref{line:combine} is distributed as $\L(\epsilon_0 + 2\epsilon_1 + \ldots 2^{b-1} \epsilon_{b-1} \,|\, E_{r,c})$, and that $r'$ is the vector of remaining row sums, and $c'$ is the vector of remaining column sums given the first $b-1$ bits of the entries of the table. 
Then, the $b$-th iteration of the algorithm generates $\L(\epsilon_0' \,|\, E_{r',c'})$ which is distributed as $\L(\epsilon_b \,|\, E_{r,c}, \epsilon_0,\ldots,\epsilon_{b-1})$. 
By Theorem~\ref{thm:pdc}, the expression for $t$ in Line~\ref{line:combine} is thus distributed as ${\L(\epsilon_0 + 2\epsilon_1 + \ldots 2^{b} \epsilon_{b} \,|\, E_{r,c})}$. 

After at most $\lfloor \log_2(M)\rfloor+1$ iterations, where $M$ is the largest row sum or column sum, all row sums and column sums will be zero, at which point the algorithm terminates and returns the current table $t$. 
\end{proof}

%% file: ctablesFixedRows.tex
% !TEX root = ctables.tex

\section{$2$ by $n$ tables}\label{sect:two}

%The special case of $2\times n$ contingency tables has received particular attention in the literature, as it is relatively simple while still being interesting---many statistical applications of contingency tables involve an axis with only two categories (male/female, test/control, etc).  
In this section, $N$ denotes the sum of all entries in the table. 

Dyer and Greenhill~\cite{dyergreenhill} described a $O(n^2\log N)$ asymptotically uniform MCMC algorithm based on updating a $2\times2$ submatrix at each step. Kijima and Matsui~\cite{kitajimamatsui} adapted the same chain using coupling from the past to obtain an exactly uniform sampling algorithm at the cost of an increased run time of $O(n^3\log N)$. In this section, we will show that Algorithm~\ref{PDC DSH 2 by n}, which is also exactly uniform sampling, runs with an expected $O(1 / p_n)$ number of rejections, where $p_n$ is a certain probability that a sum of independent uniform random variables lie in a certain region.

\begin{proof}[Proof of Theorem~\ref{two:row:theorem}]
Let $\xi$ be a $2\times n$ $(r,c)$-contingency table, and let $\xi'=(\xi_{11},\xi_{12},\ldots,\xi_{1,n-1})$; that is, the top row without the last entry. Since $r$ and $c$ are fixed, there is a bijective relationship between $\xi$ and $\xi'$; each determines the other. Then,
\[
\P\big[x=\xi\big]=\P\big[x_{11}=\xi_{11},x_{12}=\xi_{12},\ldots,x_{1,n-1}=\xi_{1,n-1}\big]=\frac1{(c_1+1)(c_2+1)\cdots(c_{n-1}+1)}.
\]
This does not depend on $\xi$, so $x$ is uniform restricted to $(r,c)$-contingency tabes.

%Sampling t
The first row of entries $x_{11}$, \ldots, $x_{1,n-1}$ %takes $O(n)$ steps, which 
is accepted if and only if $x_{1n}=r_1-x_{11}-\cdots-x_{1,n-1}$ lies between $0$ and $c_n$, which occurs with probability $p_n =  \P(U_1 + \ldots + U_{n-1} \in [r_1-c_n, r_1])$.
\end{proof}

\begin{theorem}\label{two:row:theorem:real}
Let $U_1, U_2, \ldots, U_{n-1}$ denote independent uniform random variables, with $U_j$ uniform over the continuous interval $[0,c_j]$, $j=1,\ldots, n-1$, and define
\[ p_n := \P\big(U_1 + \ldots + U_{n-1} \in [r_1-c_n, r_1]\big). \]
Algorithm~\ref{PDC DSH 2 by n real} produces a uniformly random $2\times n$ real--valued $(r,c)$-contingency table, with expected number of rejections $O(1/p_n)$. 
%The expected runtime of Algorithm~\ref{PDC DSH 2 by n} is $O(n/p_n)$. %produces a uniformly random $2\times n$ $(r,c)$-contingency table in expected time $O(n / p_n)$.
\end{theorem}

\begin{corollary}\label{equal density 2 by n}
When the row sums are equal and the column sums are equal, the expected number of rejections before Algorithm~\ref{PDC DSH 2 by n real} terminates is $O(n^{1/2})$.  
\end{corollary}

%Next, we show for which values of $c_1, \ldots, c_n$ and $r_2$, that $p_n = O(1/\sqrt{n})$.  %the probability that $x_{21}\in[0,c_1]$ is $O(1/\sqrt{n}).$ %, i.e, the rejection probability is $O(1/\sqrt{n})$.
%We now make several assumptions which guarantee that this probability is $O(1/\sqrt{n})$.  
We observe that the algorithm runs quickly when $r_1\approx \E[U_1+\cdots+U_n]=N/2$, i.e., the two row sums are similar in size, and also when $c_n$ is large. Since we can arbitrarily reorder the columns and choose $c_n$ to be the largest column sum, it follows that having a skewed distribution of column sums and an even distribution of row sums is advantageous to runtime.

Denote by $\Phi(\cdot)$ the cumulative distribution function of the standard normal distribution.

\begin{corollary}
Suppose $U_1 + \ldots + U_{n-1}$ satisfies the central limit theorem.
Assume there exists some $t\in\R$ such that, as $n\rightarrow\infty$, we have 
\[ \frac{r_1 - c_n - \frac{c_1 + \ldots + c_{n-1}}{2}}{\sqrt{\frac{c_1^2+\ldots c_{n-1}^2}{12}}} \to  t.
\]
Let $c_n' = c_n / \sqrt{\frac{c_1^2+\ldots c_{n-1}^2}{12}}$.  
Then, asymptotically as $n\rightarrow\infty$, 
the expected number of rejections before Algorithm~\ref{PDC DSH 2 by n} terminates is $O(\Phi(t+c_n') - \Phi(t))$.
\end{corollary}

\begin{proof}
Letting $Z$ denote a standard normal random variable, we have
\[ \P(U_1 + \ldots + U_{n-1} \in [r_1-c_n,r_1]) \sim \P(Z \in [t, t+c_n']).\qedhere \]
\end{proof}

\begin{corollary}\label{large c1}
Suppose $U_1 + \ldots + U_{n-1}$ satisfies the central limit theorem.
Suppose $c_n' \to \lambda \in (0,\infty].$ 
 Then %Algorithm~\ref{PDC DSH 2 by n} produces a uniform random $2\times n$ $(r,c)$-contingency table in expected time $O(n)$.
 the expected number of rejections before Algorithm~\ref{PDC DSH 2 by n} terminates is $O(1)$.
\end{corollary}

Corollary~\ref{large c1} says that when the square of the largest column sum dominates the sum of squares of the remaining column sums, then the majority of the uncertainty is in column~1, which is handled optimally by PDC.

%\begin{corollary}\label{equal density 2 by n}
%Suppose the column sums are all equal, and the row sums are all equal.  Then 
%Algorithm~\ref{PDC DSH 2 by n} produces a uniformly random $2\times n$ $(r,c)$-contingency table in expected time $O(n^{3/2}).$
%the expected number of rejections before Algorithm~\ref{PDC DSH 2 by n} terminates is $O(n^{3/2})$.
%\end{corollary}

\begin{proof}[Proof of Corollary~\ref{equal density 2 by n}]
In this case, we have $c_1' = c / \sqrt{c^2(n-1)/12} = O(1/\sqrt{n})$, hence we have the acceptance probability $p_n = O(1/\sqrt{n})$.  
\end{proof}

The following table demonstrates that under the conditions of Corollary~\ref{equal density 2 by n}, i.e., equal row sums and columns sums, the expected number of rejections grows like $O(\sqrt{n})$, and the expected runtime grows like $O(n^{3/2} \log N)$.

\begin{table}[H]
\centering
\begin{tabular}{|c|c|c|c|c|}
\hline
Rows & Columns & Density & Rejections & Runtime \\
% NOTE: just guessing some parameters that might be interesting, delete/amend as necessary
\hline
%2 & 2 & 5 & 10.0$^\ast$ & 2.84\textmu s$^\ast$ & 0$^\ast$ & 269ns$^\ast$ \\
%2 & 10 & 5 & 24.5$^\ast$ & 12.3\textmu s$^\ast$ & $^\ast$ & $^\ast$ \\
%2 & 100 & 5 & 78.4$^\ast$ & 284\textmu s$^\ast$ & $^\ast$ & $^\ast$ \\
%2 & 1000 & 5 & 249767$^\dagger$ & 7.36ms$^\dagger$ & $^\dagger$ & $^\dagger$ \\
2 & 2 & 5 & 0$^\ast$ & 269ns$^\ast$ \\
2 & 10 & 5 & 1.32$^\ast$ & 1.13\textmu s$^\ast$ \\
2 & 100 & 5 & 6.22$^\ast$ & 23.0\textmu s$^\ast$ \\
2 & 1000 & 5 & 21.6$^\dagger$ & 665\textmu s$^\dagger$ \\
2 & $10^4$ & 5 & 69.3$^\dagger$ & 19.7ms$^\dagger$ \\
2 & $10^5$ & 5 & 199$^\dagger$ & 580ms$^\dagger$ \\
2 & $10^6$ & 5 & 755$^\ddagger$ & 23.7s$^\ddagger$ \\

\hline

%\begin{tabular}{|c|c|c|c|}
%\hline
%Rows & Columns & Density & Rejections \\
%\hline
%2 & 5 & 5 & 1.66617$^\ast$ \\
%2 & 10 & 5 & 2.3046$^\ast$  \\
%2 & 15 & 5 & & 2.8424$^\ast$ \\
%2 & 100 & 5 & 7.1724$^\ast$ \\
%2 & 1000 & 5 & 22.4236$^\ast$ \\
%2 & $10^4$ & 5 & 72.632$^\dagger$ \\
%2 & $10^5$ & 5 & 309.3$^\ddagger$ \\
%2 & $10^6$ & 5 & 404$^{\ast\ast}$ \\
\end{tabular}

\caption{Simluated runtime under Algorithm 4 for sampling contingency tables with homogeneous row and column sums, compared to optimised rejection sampling where columns are picked using a discrete uniform random variable. The symbols $\ast$, $\dagger$ and $\ddagger$ denote averages over a sample of size $10^6$, 1000 and 1 respectively.
As predicted analytically, the number of rejections grows as $O(\sqrt{n})$ while the runtime grows as $O(n^{3/2} \log N)$.}

\end{table}

%% file: ctablesRealBinary.tex
% !TEX root = ctables.tex

\section{Other Tables}
\label{sect:other}

One can more generally sample from a table having independent entries with marginal distributions $\L(X_{1,1})$, $\L(X_{1,2})$, $\ldots$, $\L(X_{m,n})$, i.e., 
\begin{equation}\label{eq:general}
 \L(X) = \L( (X_{i,j})_{1\leq i \leq m, 1 \leq j \leq n} | E).
 \end{equation}
If the rejection probabilities can be computed, then we can apply a variation of Algorithm~\ref{mainalgorithm}, and possibly also a variation of Algorithm~\ref{mainalgorithm real}. %, for the discrete and continuous cases, respectively. 
%However, the row-rejection probabilities are not easily computable, whereas the

It is sometimes possible to circumvent the column-rejection probabilities. %, as in Algorithm~\ref{binary algorithm}. %, which we now present in a general form. 
We now state Algorithm~\ref{alg:iid general}, which is a general procedure of independent interest that samples from a conditional distribution of the form \[\L\left((X_1, X_2, \ldots, X_n) \middle| \sum_{i=1}^n X_i = k\right);\] see Lemma~\ref{PDC sums} for the explicit form of the rejection probability $t(a)$.  %, which is a general sampling algorithmfor continuous random variables analogous to Algorithm~\ref{alg:iid}, and prove both in Lemma~\ref{PDC sums} below.

\begin{algorithm}[H]
\caption{Random generation from $\L((X_1, X_2, \ldots, X_n) | \sum_{i=1}^n X_i = k)$}
\label{alg:iid general}
\begin{algorithmic}[1]
%\STATE \textbf{Procedure:} Constrained\_Sum\_Vector$(\L(X_1), \ldots, \L(X_n), k)$
\STATE \textbf{Assume:} $\L(X_1), \L(X_2), \ldots, \L(X_n)$ are independent and either all discrete with $\P(\sum_{i=1}^n X_i = k)>0,$
 or $\L(\sum_{i=1}^n X_i)$ has a bounded density with $f_{\sum_{i=1}^n X_i}(k)>0$. 
\IF {$n = 1$}
\RETURN $k$
\ENDIF
\STATE let $r$ be any value in $\{1,\ldots, n\}$.  
%such that the entropy of $(X_1, \ldots, X_r)$ is approximately the same as the entropy of $(X_{r+1},\ldots,X_n)$.
\FOR {$i=1,\ldots, r$}
    \STATE generate $X_{i}$ from $\L(X_i)$.  %from Geometric$(p)$
\ENDFOR
\STATE let $a \equiv (X_1, \ldots, X_r)$
\STATE let $s \equiv	 \sum_{i=1}^r X_i$.  
  \IF {$U \geq t(a)$  \ (See Lemma~\ref{PDC sums})} \label{rejection:step} %\frac{ \P\left(\sum_{i=r+1}^n X_i = k - s\, |\, a \right) }{\max_\eta \P(\sum_{i=r+1}^n X_i = \eta) }$}
    \STATE \textbf{restart}
  \ELSE
    \STATE Recursively call Algorithm~\ref{alg:iid general} on $\mathcal{L}(X_{r+1}),\ldots,\mathcal{L}(X_n)$, with target sum $k-s$, and set $(X_{r+1}, \ldots, X_n)$ equal to the return value.
%    \STATE Let $(X_{r+1}, \ldots, X_n)$ = Constrained\_Sum\_Vector$(\mathcal{L}(X_{r+1}),\ldots,\mathcal{L}(X_n), k-s)$
    \RETURN $(X_1, \ldots, X_{n})$.\label{return:value}
  \ENDIF
\end{algorithmic}
\end{algorithm}

\begin{lemma}\label{PDC sums}
%Suppose $X_1, \ldots, X_n$ are i.i.d. random variables, with $\sum_{i=1}^n X_i$ either discrete or having a density. 
%Then 
Suppose for each $a, b, \in \{1,\ldots,n\}$, $a<b$, either
\begin{enumerate}
\item $\L\left(\sum_{i=a}^b X_i\right)$ is discrete and
\begin{equation}\label{t discrete}
 t(a) = \frac{ \P\left(\sum_{i=r+1}^n X_i = k - \sum_{i=1}^r X_i | X_1, \ldots, X_r \right) }{\max_\eta \P(\sum_{i=r+1}^n X_i = \eta) }; 
 \end{equation}
or
\item $\L\left(\sum_{i=a}^b X_i\right)$ has a bounded density, denoted by $f_{a,b}$, and
\begin{equation}\label{t continuous}
 t(a) = \frac{ f_{r+1,n}\left(k - \sum_{i=1}^r X_i | X_1, \ldots, X_r \right) }{\max_\eta f_{r+1,n}(\eta) }. 
 \end{equation}
\end{enumerate}
Then Algorithm~\ref{alg:iid general} samples from $\L((X_1, X_2, \ldots, X_n) | \sum_{i=1}^n X_i = k).$
\end{lemma}
\begin{proof}
The rejection probability $t(a)$ is defined, depending on the setting, by Equation~\eqref{t discrete} or Equation~\eqref{t continuous}, so that once the algorithm passes the rejection step in Line~\ref{rejection:step}, then for any $1 \leq b \leq n$, the vector $(X_1,\ldots, X_b)$ has distribution $\L(A | h(A,B)=1)$, where $A = (X_1,\ldots, X_b)$ and $h(A,B) = \mathbbm{1}(\sum_{i=1}^n X_i = k)$.  Let $a$ denote the observed value in this stage.

We now use induction on $n$.   
When $n=1$, we take $A = (X_1)$ and $B = \emptyset$, then Algorithm~\ref{alg:iid general} with input 
% Procedure~Constrained\_Sum\_Vector
$(\mathcal{L}(X_1), k)$ returns the value of the input target sum $k$ for any $k \in \text{range}(X_1)$, which has distribution $\mathcal{L}(X_1 | X_1 = k)$.

Assume, for all $1 \leq b < n,$ Algorithm~\ref{alg:iid general} with input 
% Constrained\_Sum\_Vector
$(\mathcal{L}(X_{b+1}),\ldots,\mathcal{L}(X_n), \ell)$ returns a sample from $\mathcal{L}(X_{b+1}, \ldots X_n | h(a,B) = \ell),$ for any $\ell \in \text{range}(\sum_{i=1}^n X_i)$; i.e., it returns a sample from $\L(B | h(a,B) = 1)$, say $b$, where $B = (X_{b+1}, \ldots, X_n)$.

Hence, each time Algorithm~\ref{alg:iid general} 
% the function Constrained\_Sum\_Vector 
is called, it first generates a sample from distribution ${\L(A | h(A,B)=1)}$, and then the return value of the recursive call in Line~\ref{return:value} returns a sample from $\L(B | h(a,B)=1)$.  
%By a modification to 
By Lemma~\ref{pdc:rejection}, $(a,b)$ is a sample from $\mathcal{L}((A,B) | h(A,B)=1)$.  
\end{proof}

%It may also be 
In the case where computing the row rejection probabilities after the generation of each column is not practical, %, this is more the exception to the rule, not to mention the typically harder row-rejection probabilities. 
%In this case, 
we recommend independent sampling of columns $1, \ldots, n-1$ all at once, with a single application of PDC deterministic second half for the generation of the final column. 
%; that is, sample the first $n-1$ columns independently, and perform only one final row rejection. 
%Naturally, the expected number of rejections using this approach is higher than applying a row rejection for each column one at a time, however, t
This approach is more widely applicable, as it requires very little information about the marginal distribution of each entry.  %the entries. 
% is another PDC algorithm which is often useful, and often ideal in terms of cost, and equally applicable to discrete and continuous random variables. 

Let $X_{i,n}'$ denote the random variable with distribution $\L\left(X_{i,n} \middle| \sum_{i=1}^m X_{i,n}\right)$.
In the following algorithm, the function used for the rejection probability, when $X_{i,n}'$ is discrete, is given by 
\[ h(i,\L(X_{i,n}'),k) = \Pr( X_{i,n}' = r_i - k), \]
and when $X_{i,n}'$ is continuous with bounded density $f_{X_{i,n}'}$, is given by
\[ h(i,\L(X_{i,n}'),k) = f_{X_{i,n}'}(r_i - k). \]
Thus, for Algorithm~\ref{mainalgorithm continuous} below, we simply need to be able to compute the distribution $\L(X_{i,n}')$ and find its mode, for $i=1,2,\ldots,m$. 
%This is an example of PDC \emph{deterministic second half}, see~\cite{PDC, PDCDSH}.

\begin{algorithm}[H]
\caption{Generating a random variate from $\L(X)$ specified in Equation~\eqref{eq:general}.}
\label{mainalgorithm continuous}
\begin{algorithmic}[1]
\FOR {$j=1,\ldots,n-1$}
 % \STATE let $C_j$ be the output of Algorithm~\ref{lag:iid general} with input $(\L(X_{1,j}, X_{2,j},\ldots,X_{n,j})$ %$X_{1}, \ldots, X_{m}$ denote iid Exponential($\frac{m}{m+c_j}$) random variables
  \STATE let $C_j$ denote the return value of Algorithm~\ref{alg:iid general} using input %= \text{Constrainted\_Sum\_Vector}
  $(\L(X_{1,j}),\ldots,\L(X_{m,j}), c_j).$ %(see Algorithm~\ref{alg:iid general})
%  \STATE $C_j = \text{Constrainted\_Sum\_Vector}(\L(X_{1j}),\ldots,\L(X_{mj}), c_j)$ (see Algorithm~\ref{alg:iid general})
%  \FOR {$i=2,\ldots,m$}
%    \STATE generate $x_{ij}$ from Geometric($\frac{m}{m+c_j}$)
%  \ENDFOR
%\STATE let $x_{1j} = c_j - \sum_{i=2}^m x_{ij}$
%\STATE \textbf{with probability} $1-\big(\tfrac{c_j}{m+c_j}\big)^{x_{1j}}\mathbbm{1}_{\{x_{1j}\ge0\}}$ \textbf{restart} from Line 2
\ENDFOR
%\FOR {$i=1,\ldots,m$}
  \STATE let $x_{i,n} = r_i - \sum_{j=1}^{n-1} x_{i,j}$ for $i=1,\ldots,m$.
  \STATE \textbf{if} {$U > \prod_{i=1}^m \frac{h(i,\L(X_{i,n}'),x_{i,n})}{\sup_k h(i,\L(X_{i,n}'),k)}$} \textbf{restart} from Line 1
%\ENDFOR
\STATE \textbf{return} $x$
\end{algorithmic}
\end{algorithm}

\begin{proposition}
Algorithm~\ref{mainalgorithm continuous} samples points according to the distribution in~\eqref{eq:general}. % real--valued points uniformly in $E$.
\end{proposition}

The proof follows analogously to Lemma~\ref{uniform:lemma}, and uses the conditional independence of the rows given the column sums are satisfied.  %, with probability mass functions changed to probability density functions.  

In the most general case when even the columns cannot be simulated using Algorithm~\ref{alg:iid general} or another variant, 
%Instead of the uniform measure over contingency tables, o
%One can more generally sample from a table having independent entries with marginal distributions $\L(X_{11})$, $\L(X_{12})$, $\ldots$, $\L(X_{mn})$.  
%In this case, we do not assume to be able to quickly calculate the rejection probabilities, % or probability density functions of sums of these random variables, 
%as is required in Algorithm~\ref{alg:iid general},  % thus we cannot quickly generate constrained columns.
%Instead, 
we apply PDC deterministic second half to both the columns and the rows, which simply demands in the continuous case that there is at least one random variable with a bounded density per column (resp., row).  
In Algorithm~\ref{mainalgorithm general} below, each column has a rejection function $t_j$, which is either the normalization of the probability mass function 
\[ t_j = \frac{\P(X_{i_j,j} = a)}{\max_\ell \P(X_{i_j,j}=\ell)} \]
 or the normalization of the probability density function 
 \[ t_j = \frac{f_{X_{i_j,j}}(a)}{\sup_\ell f_{X_{i_j,j}}(\ell)}. \]
There is also a row rejection function $s_i$.  Let $X_{i,j}'$ denote the random variable with distribution $\L\left(X_{i,j} \middle| \sum_{\ell=1}^m X_{\ell,j}\right)$.  When $\L(X_{i,j}')$ is discrete, we have
\[ s_i(a) = \frac{\P(X_{i,j}' = a)}{\max_\ell \P(X_{i,j}'=\ell)} \]
and when $\L(X_{i,j}')$ is continuous, we have
\[ s_i(a) = \frac{f_{X_{i,j}'}(a)}{\sup_\ell f_{X_{i,j}'}(\ell)}. \]

\begin{algorithm}[H]
\caption{Generating $\L(\,(X_{1,1}, X_{1,2},\ldots, X_{m,n})\, |\, E)$}
\label{mainalgorithm general}
\begin{algorithmic}[1]
\STATE let $j' \in \{1,\ldots, n\}$ denote a column.  \label{line:one}
\FOR {$j\in \{1,2,\ldots,n\} \setminus \{j'\}$}\label{line:two}
  \STATE let $i_j \in \{1,\ldots, m\}$ denote a row in the $j$-th column
  \FOR {$i=\{1,2,\ldots,m\} \setminus \{i_j\}$}
    \STATE generate $x_{ij}$ from $\L(X_{ij})$
  \ENDFOR
\STATE let $x_{i_jj} = c_j - \sum_{i\neq i_j} x_{ij}$
\STATE \textbf{with probability} $1-t_{j}(x_{i_j j})$ \textbf{restart} from Line~\ref{line:two}
\ENDFOR
\FOR {$i=1,\ldots,m$}
  \STATE let $x_{ij'} = r_i - \sum_{j\neq j'} x_{ij}$
  \STATE \textbf{with probability} $1-s_{i}(x_{ij'})$ \textbf{restart} from Line~\ref{line:two}
  \STATE \textbf{if} {$x_{ij'} < 0$} \textbf{restart} from Line~\ref{line:one}
\ENDFOR
\STATE \textbf{return} $x$
\end{algorithmic}
\end{algorithm}

\begin{proposition}
Algorithm~\ref{mainalgorithm general} samples points according to the distribution in~\eqref{eq:general}. % real--valued points uniformly in $E$.
\end{proposition}

% end ignore

% end ignore

\ignore{
% too many algorithm blocks, let's have a paragraph instead
\begin{algorithm}[H]
\caption{$(X_1, X_2, \ldots, X_{n}|\, T = k)$}
\begin{algorithmic}
%\Procedure {BellmanKalaba}{$G$, $u$, $l$, $p$}
\STATE \assume at least one $X_i$ is a continuous random variable with bounded density
\STATE Select an index of any of the continuous random variables, say $I$.
\STATE Sample $x_1, \ldots, x_{I-1}, x_{I+1}, \ldots, x_n$
\STATE $t_I \leftarrow k - t_n^I$
\STATE let $M = \sup_\ell \ f_{X_I}(\ell)$.
\IF {$t_I \in \range(X_I)$ }
	\IF {$u\, M < f_{X_I}(t_I)$}
		\STATE $x_I \leftarrow t_I$
		\STATE \return $x$
%	\Else
%	\State \restart
	\ENDIF
\ELSE
\STATE \restart
\ENDIF
%\EndProcedure
\end{algorithmic}\label{PDC continuous}
\end{algorithm}
}% end ignore

% end ignore

%% file: ctablesAcknowledgements.tex
% !TEX root = ctables.tex

\section{Acknowledgements}

The authors gratefully acknowledge helpful suggestions by Chris Anderson, Richard Arratia, Jesus de Loera, and also Igor Pak for help with the literature.